\documentclass[]{interact}

\usepackage{epstopdf}
\usepackage[caption=false]{subfig}
\usepackage{algorithm}
\usepackage{algorithmicx}
\usepackage{algpseudocode}
\usepackage{multirow}

\usepackage{threeparttable}

\usepackage[numbers,sort&compress]{natbib}
\bibpunct[, ]{[}{]}{,}{n}{,}{,}

\theoremstyle{plain}
\newtheorem{theorem}{Theorem}[section]
\newtheorem{lemma}[theorem]{Lemma}
\newtheorem{corollary}[theorem]{Corollary}

\theoremstyle{definition}

\newtheorem{example}[theorem]{Example}

\theoremstyle{remark}
\newtheorem{remark}{Remark}

\begin{document}


\title{Kaczmarz-Type Methods for Solving Matrix Equations}

\author{
\name{Weiguo Li\textsuperscript{a}\thanks{CONTACT Weiguo Li. Email: liwg@upc.edu.cn}, Wendi Bao\textsuperscript{a}\thanks{CONTACT Wendi Bao. Email: baowd@upc.edu.cn}, Lili Xing\textsuperscript{a}\thanks{CONTACT Lili Xing. Email: xinglily2010@upc.edu.cn}, and Zhiwei Guo\textsuperscript{a}}\thanks{CONTACT Zhiwei Guo. Email: gzw\_13605278246@163.com}
\affil{\textsuperscript{a}College of Science, China University of Petroleum, Qingdao 266580, P .R. China}
}

\maketitle

\begin{abstract}
In this paper, several Kaczmarz-type numerical methods for solving the matrix equation $AX=B$ and $XA=C$ are proposed, where the coefficient matrix $A$ may be full rank or rank deficient. These methods are iterative methods without matrix multiplication. Theoretically, the convergence of these methods is proved. The numerical results show that these methods are more efficient than iterative methods involving matrix multiplication for high-dimensional matrices.
\end{abstract}


\section{Introduction}

Matrix equations play a significant role in various mathematical fields such as differential
equations, algebra, probability and statistics, calculus of several variables, biological sciences, economics and management
studies. In applied mathematics, the development and analysis of the various characteristics of iterative algorithms for solving the matrix equation is an active area of present research. See some recent references, such as \cite{YY10,LX18,HL23,ZHL23}.

Many methods among these frequently use the matrix-matrix product operation, which consumes a lot of computing time. In this paper, Kaczmarz and coordinate descend methods \cite{K37,WSJ15} are used to obtain solutions of matrix equations by the product of matrix and vector. All the results in this paper hold in the complex field. But for the sake of simplicity, we only discuss it in the real number field.

In this paper, we denote $A^T$,  $A^+$, $r(A)$, $R(A)$, $\|A\|_F$, and $\langle A, B\rangle= trace(A^TB)$ as the transpose, the Moore-Penrose generalized inverse (abbreviated as MP inverse), the rank of $A$, the column space of $A$, the Frobenius
norm of $A$ and the inner product of two matrices $A$ and $B$, respectively. We indicate by $I_n$ the identity matrix in $R^{n\times n}$. In addition, for a given matrix $G=(g_{ij})\in R^{m\times n}$, $G_{i,:}$, $G_{:,j}$ and $\sigma_{min}(G)$, are used to denote its $i$th row, $j$th column and the smallest nonzero singular value of $G$ respectively. Let $E_k$ denote the expected value conditional on the first k iterations, that is,
$$E_k[\cdot]=E[\cdot|j_0,j_1,...,j_{k-1}],$$ where $j_s(s=0,1,...,k-1)$ is the column chosen at the $s$th iteration.

The organization of this paper is as follows. In Section 2, we discuss the Kaczmarz method to solve the consistent matrix equations $AX=B$ and $XA=C$. In Section 3, we discuss the coordinate descent (CD) method for solving the inconsistent matrix equations $AX=B$ and $XA=C$ in the case of full column or row rank. In Section 4, we discuss the extended Kaczmarz methods and the extended coordinate descent methods for solving the inconsistent matrix equations $AX=B$ and $XA=C$. In Section 5, using the advantages of the row by row or column by column orthogonal projection of Kaczmarz-type iteration, some recursive strategies for solving matrix equations are given. In Section 6, some numerical examples are provided to illustrate the effectiveness of our new methods. Finally, some brief concluding remarks are described in Section 7.

\section{Solving Consistent Matrix Equation $AX=B$ and $XA=C$ by Kaczmarz Method}

Considering the following consistent matrix equation
\begin{equation}\label{e201}
	AX=B,
\end{equation}
where $A\in R^{m\times n}$ and $B\in R^{m\times p}$. So there exists an $X^*\in R^{n\times p}$ satisfying $AX^*=B$. In general, the equation (\ref{e201}) has multiple solutions. Now we try to find its minimal $F$-norm solution $A^+B$ by Kaczmarz method.

The classical Kaczmarz method which was introduced in 1937 \cite{K37} is a row projection iterative algorithm
for solving a consistent system $Ax = b$ where $A\in R^{m\times n}$, $b\in R^m$ and $x\in R^n$. This method involves only a single equation per iteration as follows which converges to the least norm solution of $Ax = b$ with an initial iteration $x_0\in R(A^T)$,
\begin{equation*}\label{e202}
	x^{(k+1)}=x^{(k)}+\frac{b_i-A_{i,:}x^{(k)}}{\|A_{i,:}\|_2^2}(A_{i,:})^T, \  \ k\ge 0,
\end{equation*}
where $i=(k \ mod \ m) + 1$. If we iterate the system of linear equations $AX_{:,j}=B_{:,j}$, $j=1, \cdots, p$ simultaneously and denote $X^{(k)}=[X_{:,1}^{(k)}, X_{:,2}^{(k)}, \cdots, X_{:,p}^{(k)}]$, then we get
\begin{equation}\label{e203}
	X^{(k+1)}=X^{(k)} +\frac{(A_{i,:})^T}{\|A_{i,:}\|_2^2}(B_{i,:}-A_{i,:}X^{(k)}), \  \ k\ge 0,
\end{equation}
where $i=(k \ mod \ m) + 1$ and $A_{i,:}X^{(k+1)}=B_{i,:}$ holds, that is, $X^{(k+1)}$ is a projection of $X^{(k)}$ onto the subspace $H_i=\left\{X\in R^{n\times p}: \ A_{i,:}X=B_{i,:}\right\}$. So we obtain the following randomized Kaczmarz (RK) method for consistent matrix equation $AX=B$ (RKCAX).
\begin{algorithm}
	\leftline{\caption{RK Method for Consistent Matrix Equation $AX=B$ (RKCAX)\label{alg21}}}
	\begin{algorithmic}[1]
		\Require
		$A\in R^{m\times n}$, $B\in R^{m\times p}$, $X^{(0)}\in R^{n\times p}$, $K\in R$
		\State For $i=1:m$, $M(i)=\|A_{i,:}\|_2^2$
		\For {$k=0,1,2,\cdots, K-1$}
		\State Set $p_{row=i}=\frac{\|A_{i,:}\|_2^2}{\|A\|^2_F}$
		\State $X^{(k+1)}=X^{(k)}+\frac{(A_{i,:})^T}{M(i)}(B_{i,:}-A_{i,:}X^{(k)})$
		\EndFor
		\State Output $X^{(K)}$
	\end{algorithmic}
\end{algorithm}

The cost of each iteration of this method is $4np+n$ if the square of the row norm of $A$ has been calculated in advance.
In the following theorem, with the idea of the RK method \cite{SV09}, we will prove that iteration (\ref{e203}) will converges to the the least $F$-norm solution of $AX=B$ if $i$ is picked at random.

\begin{lemma}\label{l201}
	If $X\in R(A^T)$, we have
	\begin{equation*}\label{e204}
		\|AX\|_F^2\geq\sigma_{min}^2(A)\|X\|_F^2.
	\end{equation*}
\end{lemma}

\begin{theorem}\label{t201}
	The sequence $\{X^{(k)}\}$ generated by Algorithm \ref{alg21} starting from the initial matrix $X^{(0)}$, converges linearly to $X^*=A^+B+(I_n-A^+A)X^{(0)}$ in mean square and the following relationship holds for arbitrary $A\in R^{m\times n}$ and $B\in R^{m\times p}$
	\begin{equation}\label{e205}
		E[\|X^{(k)}-X^*\|_F^2]\leq\left(1-\frac{\sigma_{min}^2(A)}{\|A\|^2_F}\right)^k\|X^{(0)}-X^*\|_F^2,
	\end{equation}
	where the $i$th row of $A$ is selected with probability $p_{row=i}=\frac{\|A_{i,:}\|_2^2}{\|A\|^2_F}$.
\end{theorem}
\begin{proof}
	Obviously, (\ref{e203}) is equivalent to the following expression
	\begin{equation}\label{e206}
		X^{(k+1)}-X^*=X^{(k)}-X^*+\frac{(A_{i,:})^T(A_{i,:}X^*-A_{i,:}X^{(k)})}{\|A_{i,:}\|_2^2}=\left(I_n-\frac{(A_{i,:})^TA_{i,:}}{\|A_{i,:}\|_2^2}\right)(X^{(k)}-X^*),
	\end{equation}
	where $AX^*=AA^+B=B$ (consistent), that is, $A_{i,:}X^*=A_{i,:}A^+B=B_{i,:}$, $i=1,\cdots,m$.
	
	Since $X^{(0)}-X^*=A^+(AX^{(0)}-B)\in R(A^T)$, it is easy to see from (\ref{e206}) that $X^{(k)}-X^*\in R(A^T)$ for all $k=0,1,\cdots$. And we know
	\begin{align}\label{e207}
		\sum_{i=1}^m  \frac{\|A_{i,:}\|^2_2}{ \|A\|^2_F}\left\|\left(I-\frac{(A_{i,:})^TA_{i,:}}{\|A_{i,:}\|^2_2}\right)(X^{(k)}-X^*)\right\|_F^2
		=\|X^{(k)}-X^*\|_F^2 -  \frac{\|A(X^{(k)}-X^*)\|_F^2}{\|A\|^2_F},
	\end{align}
	and by Lemma \ref{l201}, we have
	\begin{equation*}\label{e208}
		\|X^{(k)}-X^*\|_F^2 -  \frac{\|A(X^{(k)}-X^*)\|_F^2}{\|A\|^2_F}\leq \left(1-\frac{\sigma^2_{\min}(A)}{\|A\|_F^2}\right)\|X^{(k)}-X^*\|_F^2.
	\end{equation*}
	Therefore, we get
	\begin{equation}\label{e209}
		E[\|X^{(k+1)}-X^*\|_F^2]\leq \left(1-\frac{\sigma^2_{\min}(A)}{\|A\|^2_F}\right)E[\|X^{(k)}-X^*\|_F^2],
	\end{equation}
	where the $i$th row of $A$ is selected with probability $p_i(A)=\frac{\|A_{i,:}\|_2^2}{\|A\|^2_F}$. From (\ref{e209}), we immediately get the convergence of sequence $\{X^{(k)}\}$.
\end{proof}

\begin{corollary}\label{c201}
	The sequence $\{X^{(k)}\}$ generated by Algorithm \ref{alg21} converges linearly to $X^*=A^+B$ in mean square if $X^{(0)}_{:,j}\in R(A^T)$, $j=1,\cdots,p$ and
	\begin{equation*}\label{e210}
		E[\|X^{(k)}-A^+B\|_F^2]\leq\left(1-\frac{\sigma_{min}^2(A)}{\|A\|^2_F}\right)^k\|X^{(0)}-A^+B\|_F^2,
	\end{equation*}
	where the $i$th row of $A$ is selected with probability $p_{row=i}=\frac{\|A_{i,:}\|_2^2}{\|A\|^2_F}$.
\end{corollary}
\begin{proof}
	Due to $X^{(0)}_{:,j}\in R(A^T)$, $j=1,\cdots,p$, we assume that $X^{(0)}=A^TY^{(0)}$. By $A^+AA^T=A^T$, we know that $X^*=A^+B+(I_n-A^+A)X^{(0)}=A^+B+X^{(0)}-A^+AA^TY^{(0)}=A^+B$.
\end{proof}

Considering the following consistent matrix equation
\begin{equation}\label{e211}
	XA=C,
\end{equation}
where $A\in R^{m\times n}$ and $C\in R^{p\times n}$. So there exists an $X^*\in R^{p\times m}$ satisfying $X^*A=C$. In general, the equation (\ref{e211}) has multiple solutions. Now we try to find its minimal $F$-norm solution $CA^+$ by Kaczmarz method.

Equation (\ref{e211}) is equivalent to the matrix equation
\begin{align}\label{212}
	A^TY=C^T,
\end{align}
where $X=Y^T$. Based on the iteration (\ref{e203}), we get the RK method for solving the consistent matrix equation $XA=C$ denoted as the RKCXA method.
\begin{equation*}\label{e213}
	X^{(k+1)}=X^{(k)}+\frac{C_{:,j}-X^{(k)}A_{:,j}}{M(j)}(A_{:,j})^T, \  \ k\ge 0,
\end{equation*}
where $j=(k \ mod \ n)+1$. This is a column projection method and the cost of each iteration of the method is $4mp+m$.

\begin{theorem}\label{t202}
	The sequence $\{X^{(k)}\}$ generated by the RKCXA method starting from the initial matrix $X^{(0)}\in R^{p\times m}$ converges linearly to $X^*=CA^++X^{(0)}(I_n-AA^+)$ in mean square and
	\begin{equation*}\label{e214}
		E[\|X^{(k)}-X^*\|_F^2]\leq\left(1-\frac{\sigma_{min}^2(A)}{\|A\|^2_F}\right)^k\|X^{(0)}-X^*\|_F^2,
	\end{equation*}
	where the $j$th column of $A$ is selected with probability $p_{col=j}=\frac{\|A_{:,j}\|_2^2}{\|A\|^2_F}$.
\end{theorem}
\begin{proof}
	Similar to the proof of Theorem \ref{t201}.
\end{proof}

\begin{corollary}\label{c202}
	The sequence $\{X^{(k)}\}$ generated by the RKCXA method. starting from the initial matrix $X^{(0)}$, converges linearly to $CA^+$ in mean square if $(X^{(0)}_{i,:})^T\in R(A),\ i=1,\cdots,p$, and the following relationship holds
	\begin{equation*}\label{215}
		E[\|X^{(k)}-CA^+\|_F^2]\leq\left(1-\frac{\sigma_{min}^2(A)}{\|A\|^2_F}\right)^k\|X^{(0)}-CA^+\|_F^2,
	\end{equation*}
	where the $j$th column of $A$ is selected with probability $p_{col=j}=\frac{\|A_{:,j}\|_2^2}{\|A\|^2_F}$.
\end{corollary}

{\bf Computing Right Inverse and Left Inverse of Matrix by Kaczmarz Method}

If the matrix $A\in R^{m\times n}$ is full row rank ($m\leq n$), then the right inverse $X$ of $A$ meets
\begin{align}\label{e216}
	AX=I_m.
\end{align}
The MP inverse $A^+=A^T(AA^T)^{-1}$ is one of the right inverse of $A$ and it is also the unique minimal $F$-norm solution of the matrix equation (\ref{e216}). Based on Algorithm \ref{alg21} ($B=I_m$), we can find all the right inverse of $A$ (full row rank) by Kaczmarz method.

\begin{corollary}\label{c203}
	The sequence $\{X^{(k)}\}$ generated by Algorithm \ref{alg21} ($B=I_m$) starting from any initial matrix $X^{(0)}\in R^{n\times m}$, converges linearly to $X^*=A^++(I_n-A^+A)X^{(0)}$ in mean square if $A$ is full row rank, and
	\begin{equation*}\label{e217}
		E[\|X^{(k)}-X^*\|_F^2]\leq\left(1-\frac{\sigma_{min}^2(A)}{\|A\|^2_F}\right)^k\|X^{(0)}-X^*\|_F^2,
	\end{equation*}
	where the $i$th row of $A$ is selected with probability $p_{row=i}=\frac{\|A_{i,:}\|_2^2}{\|A\|^2_F}$.
\end{corollary}

\begin{remark}\label{r201}
	If the matrix $A\in R^{m\times n}$ is full row rank ($m\leq n$), then the general right inverse of $A$ is
	\begin{equation*}\label{218}
		X^*=A^++(I_n-A^+A)Z
	\end{equation*}
	for arbitrary $Z\in R^{n\times m}$. See \cite{BG13}.
\end{remark}

If the matrix $A\in R^{m\times n}$ is full column rank ($m\geq n$), then the left inverse $X$ of $A$ meets
\begin{align}\label{e219}
	XA=I_n.
\end{align}
The MP inverse $A^+=(A^TA)^{-1}A^T$ is one of the left inverse of $A$ and it is also the unique minimal $F$-norm solution of the matrix equation (\ref{e219}). Based on Algorithm \ref{alg22} ($C=I_n$), we can find all the left inverse of $A$ (full column rank) by Kaczmarz method.

\begin{corollary}\label{c204}
	The sequence $\{X^{(k)}\}$ generated by the RKCXA method ($C=I_n$) starting from the initial matrix $X^{(0)}\in R^{n\times m}$, converges linearly to $X^*=A^++X^{(0)}(I_m-AA^+)$ in mean square if $A$ is full column rank, and
	\begin{equation*}\label{220}
		E[\|X^{(k)}-X^*\|_F^2]\leq\left(1-\frac{\sigma_{min}^2(A)}{\|A\|^2_F}\right)^k\|X^{(0)}-X^*\|_F^2,
	\end{equation*}
	where the $j$th column of $A$ is selected with probability $p_{col=j}=\frac{\|A_{:,j}\|_2^2}{\|A\|^2_F}$.
\end{corollary}

\begin{remark}\label{r202}
	If the matrix $A\in R^{m\times n}$ is full column rank ($m\geq n$), then the general left inverse of $A$ is
	\begin{equation*}\label{221}
		X^*=A^++Z(I_m-AA^+)
	\end{equation*}
	for arbitrary $Z\in R^{n\times m}$. See \cite{BG13}.
\end{remark}

\section{Solving the Least Square Solution of Inconsistent Matrix Equation by CD Method}
Consider the matrix equation (which may be inconsistent)
\begin{equation}\label{e301}
	AX=B,
\end{equation}
where $A\in R^{m\times n}$, $B\in R^{m\times p}$. Now we will find its unique minimal $F$-norm least square solution $X^*=A^+B$ with the CD or Gauss-Seidel (GS) method.

If a systems of linear equations $Ax=b$ is inconsistent, the CD method below is a very effective method to solve its least square solution for very large systems of linear equations.
\begin{equation*}\label{e302}
	\alpha_k=\frac{(A_{:,j})^T r^{(k)}}{\|A_{:,j}\|_2^2}, \ x^{(k+1)}_j=x^{(k)}_j+\alpha_k, \  r^{(k+1)}=r^{(k)}-\alpha_k A_{:,j}, \ \ j=(k \ mod \ n)+1,
\end{equation*}
where $x^{(0)}\in R^n$ is arbitrary and $r^{(0)}=b-Ax^{(0)}$. Applying simultaneous $p$ iterative formulae for solving $AX_{:,l}=B_{:,l},\ l=1, \cdots, p$, we get
\begin{equation}\label{e303}
	W^{(k)}=\frac{(A_{:,j})^TR^{(k)}}{\|A_{:,j}\|_2^2}, \ X^{(k+1)}_{j,:}=X^{(k)}_{j,:}+W^{(k)}, \  R^{(k+1)}=R^{(k)}-A_{:,j}W^{(k)}, \ \ j=(k \ mod \ n)+1,
\end{equation}
where $X^{(0)}\in R^{n\times p}$, $R^{(0)}=B-AX^{(0)}$. This is a column projection method and the cost of each iteration of the method is $4mp + p$ if the square of the row norm of $A$ has been calculated in advance. The randomized Gauss-Seidel algorithm for $AX=B$ is described as Algorithm \ref{alg31}.

\begin{algorithm}
	\leftline{\caption{RGS Method for Inonsistent Matrix Equation $AX=B$ (RGSIAX)\label{alg31}}}
	\begin{algorithmic}[1]
		\Require
		$A\in R^{m\times n}$, $B\in R^{m\times p}$, $X^{(0)}\in R^{n\times p}$, $R^{(0)}=B-AX^{(0)}$, $K\in R$
		\State For $j=1:n$, $M(j)=\|A_{:,j}\|_2^2$
		\For {$k=0,1,2,\cdots, K-1$}
		\State Set $p_{col=j}=\frac{\|A_{:,j}\|_2^2}{\|A\|^2_F}$
		\State $W^{(k)}=\frac{(A_{:,j})^TR^{(k)}}{\|A_{:,j}\|_2^2}, \ X^{(k+1)}_{j,:}=X^{(k)}_{j,:}+W^{(k)}, \  R^{(k+1)}=R^{(k)}-A_{:,j}*W^{(k)}$
		\EndFor
		\State Output $X^{(K)}$
	\end{algorithmic}
\end{algorithm}

Obviously, $X^*=A^+B$ is the minimal $F$-norm least square solution of the matrix equation (\ref{e301}) and we have the following conclusions.

\begin{theorem}\label{t301}
	Let $A\in R^{m\times n}$ and $B\in R^{m\times p}$. Let $X^{(k)}$ denote the $k$th iterate by \eqref{e303} with arbitrary $X^{(0)}\in R^{n\times p}$. In exact arithmetic, it holds
	\begin{equation}\label{e304}
		E[\|A(X^{(k)}-A^+B)\|_F^2]\leq\left(1-\frac{\sigma_{min}^2(A)}{\|A\|^2_F}\right)^k\|A(X^{(0)}-A^+B)\|_F^2,
	\end{equation}
	where the $j$th column of $A$ is selected with probability $p_{col=j}=\frac{\|A_{:,j}\|_2^2}{\|A\|^2_F}$.
\end{theorem}
\begin{proof}
	From \eqref{e303}, we know that
	\begin{equation*}\label{e305}
		A(X^{(k+1)}-A^+B)=A(X^{(k)}-A^+B)+\frac{A_{:,j}(A_{:,j})^TA(A^+B-X^{(k)})}{\|A_{:,j}\|_2^2}.
	\end{equation*}
	Then it holds that		
	\begin{equation*}\label{e306}
		\begin{array}{rl}
			\|A(X^{(k+1)}-A^+B)\|_F^2=&\|A(X^{(k)}-A^+B)+\frac{A_{:,j}(A_{:,j})^TA(A^+B-X^{(k)})}{\|A_{:,j}\|_2^2}\|_F^2\\				
			=&\|A(X^{(k)}-A^+B)\|_F^2+\|\frac{A_{:,j}(A_{:,j})^TA(A^+B-X^{(k)})}{\|A_{:,j}\|_2^2}\|_F^2\\
			&-\frac{2}{\|A_{:,j}\|_2^2}trace((X^{(k)}-A^+B)^{T}A^TA_{:,j}(A_{:,j})^TA(X^{(k)}-A^+B))\\
			=&\|A(X^{(k)}-A^+B)\|_F^2+\frac{\|(A_{:,j})^TA(A^+B-X^{(k)})\|_2^2}{\|A_{:,j}\|_2^2}-2\frac{\|(A_{:,j})^TA(X^{(k)}-A^+B)\|_2^2}{\|A_{:,j}\|_2^2}\\
			=&\|A(X^{(k)}-A^+B)\|_F^2-\frac{\|(A_{:,j})^TA(A^+B-X^{(k)})\|_2^2}{\|A_{:,j}\|_2^2}
		\end{array}
	\end{equation*}
	Therefore,
	\begin{equation}\label{e307}
		\begin{array}{rl}
			E[\|A(X^{(k+1)}-A^+B)\|_F^2]=&\|A(X^{(k)}-A^+B)\|_F^2-E\left[\frac{\|(A_{:,j})^TA(X^{(k)}-A^+B)\|_2^2}{\|A_{:,j}\|_2^2}\right]\\
			=&\|A(X^{(k)}-A^+B)\|_F^2-\frac{1}{\|A\|_F^2}\sum\limits_{i=1}^{n}\|(A_{:,j})^TA(X^{(k)}-A^+B)\|_2^2\\	
			=&\|A(X^{(k)}-A^+B)\|_F^2-\frac{1}{\|A\|_F^2}\|A^{T}A(X^{(k)}-A^+B)\|_2^2\\	
			\leq&\left(1-\frac{\sigma_{min}^2(A)}{\|A\|_F^2}\right)\|A(X^{(k)}-A^+B)\|_F^2.
		\end{array}
	\end{equation}
	By \eqref{e307} and induction on the iteration index $k$, we obtain the estimate \eqref{e304}.
\end{proof}

\begin{remark}\label{r301}
	If $A$ is full column rank, Theorem \ref{t301} implies that $X^{(k)}$ converges linearly in expectation to $A^+B$. Otherwise, the sequence does not necessarily have a limit.
\end{remark}

Similarly, we can also consider the least square solution of the following matrix equation (which may be inconsistent) by the CD method
\begin{equation}\label{e308}
	XA=C,
\end{equation}
where $A\in R^{m\times n}$ and $C\in R^{p\times n}$. Similar to (\ref{e303}), we can get the RGS method for inconsistent matrix equation $XA=C$ denoted by the RGSIXA method.
\begin{equation}\label{e309}
	U^{(k)}=\frac{R^{(k)}(A_{i,:})^T}{\|A_{i,:}\|_2^2}, \ X^{(k+1)}_{:,i}=X^{(k)}_{:,i}+U^{(k)}, \  R^{(k+1)}=R^{(k)}-U^{(k)}A_{i,:},
\end{equation}
where $X^{(0)}\in R^{p\times m}$, $R^{(0)}=C-X^{(k)}A$ and $i=(k \ mod \ m)+1$. This is a row projection method and the cost of each iteration of the method is $4np + p$.

Now we will find its minimal least square solution $X^*= CA^+$ of the matrix equation (\ref{e308}) with the CD (or GS) method.

\begin{theorem}\label{t302}
	Let $A\in R^{m\times n}$ and $C\in R^{p\times n}$. Let $X^{(k)}$ denote the $k$th iterate by \eqref{e309} with arbitrary $X^{(0)}\in R^{p\times m}$. In exact arithmetic, it holds
	\begin{equation*}\label{e310}			
		E[\|(X^{(k)}-CA^+)A\|_F^2]\leq\left(1-\frac{\sigma_{min}^2(A)}{\|A\|^2_F}\right)^k\|(X^{(0)}-CA^+)A\|_F^2,
	\end{equation*}
	where the $i$th row of $A$ is selected with probability $p_{row=i}=\frac{\|A_{i,:}\|_2^2}{\|A\|^2_F}$, and
	\begin{equation*}\label{311}
		X^*=CA^+=\arg\min\limits_{X\in R^{p\times m}}\|XA-C\|_F
	\end{equation*}
\end{theorem}
The proof is similar to that of Theorem \ref{t301}.

\begin{remark}
	If $A$ is full row rank, Theorem \ref{t302} implies that $X^{(k)}$ converges linearly in expectation to $CA^+$. Otherwise, the sequence does not necessarily have a limit.
\end{remark}

{\bf Computing Right Inverse and Left Inverse of Matrix by RGS Method}

If the matrix $A\in R^{m\times n}$ is full column rank ($m\geq n$), then $X^*=(A^TA)^{-1}A^T$ is a left inverse of $A$ and $X^*$ is also the the unique least square solution of the matrix equation
\begin{align*}\label{e312}
	AX=I_m.
\end{align*}

Based on Algorithm \ref{alg31} ($B=I_m$), we can solve $X^*$ ($A$ is full column rank) by RGS method.

\begin{corollary}\label{c301}
	Let $A$ be a full row rank. The sequence $\{X^{(k)}\}$ generated by Algorithm \ref{alg31} starting from arbitrary initial matrix $X^{(0)}\in R^{n\times m}$, converges linearly to $X^*$ in mean square ($B=I_m$), and it holds
	\begin{equation}\label{e313}
		E[\|X^{(k)}-X^*\|_F^2]\leq C\left(1-\frac{\sigma_{min}^2(A)}{\|A\|^2_F}\right)^k\|X^{(0)}-X^*\|_F^2,
	\end{equation}
	where $C\le\frac{\sigma^2_{max}(A)}{\sigma^2_{min}(A)}$ is a constant and the $j$th column of $A$ is selected with probability $p_j=\frac{\|A_{:,j}\|_2^2}{\|A\|^2_F}$.
\end{corollary}

If the matrix $A\in R^{m\times n}$ is full row rank ($m\leq n$), then $X^*=A^T(AA^T)^{-1}$ is a right inverse of $A$ and also the the unique least square solution of the matrix equation
\begin{align*}\label{e314}
	XA=I_n.
\end{align*}
Based on the RGSIXA method ($C=I_n$), we can solve $X^*$ ($A$ is full row rank) by RGS method.

\begin{corollary}\label{c302}
	Let $A$ be full row rank. The sequence $\{X^{(k)}\}$ generated by the iteration (\ref{e309}) starting from arbitrary initial matrix $X^{(0)}\in R^{n\times m}$, converges linearly to $X^*$ in mean square ($C=I_n$) and it holds
	\begin{equation*}\label{e315}
		E[\|X^{(k)}-X^*\|_F^2]\leq C \left(1-\frac{\sigma_{min}^2(A)}{\|A\|^2_F}\right)^k\|X^{(0)}-X^*\|_F^2,
	\end{equation*}
	where $C\le\frac{\sigma^2_{max}(A)}{\sigma^2_{min}(A)}$ is a constant and the $i$th row of $A$ is selected with probability $p_i=\frac{\|A_{i,:}\|_2^2}{\|A\|^2_F}$.
\end{corollary}

\section{Solving Matrix Equation $AX=B$ and $XA=C$ by REK and REGS Methods}
The extended Kaczmarz method \cite{D10} and extended CD method \cite{MN15} are applicable to all kinds of systems of linear equations (consistent or inconsistent, over-determined or under-determined, the coefficient matrix $A$ has full rank or not), we can use these two methods to compute MP inverse of any matrix $A$.

It is well known that if $A\in R^{m\times n}$ is a general matrix then it does not necessarily have matrix $B$ so that $AB=BA=I$. But there exists a canonical generalized inverse, called the Moore-Penrose (MP) inverse and denoted by $A^+$, which is uniquely determined by $A$ that satisfies the following Penrose equation.
$$AA^+A=A, \ \  A^+AA^+=A^+, \ \ (AA^+)^T=AA^+, \ \ (A^+A)^T=A^+A.$$
Further, as Penrose showed in \cite{P56} (see also \cite{W15}, \cite{M90} and \cite{SMW21} ), the MP inverse satisfies the following inequalities: for all $X\in R^{n\times p}$, $B\in R^{m\times p}$
\begin{equation}\label{e401}
	\|AX-B\|_2\geq\|AA^+B-B\|_2
\end{equation}
with equality occurring in (\ref{e401}) if and only if $X=A^+B + (I_n-A^+A)L$
where $L\in R^{n\times p}$ is arbitrary; and
\begin{equation}\label{e402}
	\|A^+B+(I_n-A^+A)L\|_2\geq\|A^+B\|_2
\end{equation}
with equality occurring in (\ref{e402}) if and only if $(I_n-A^+A)L=0$. (The only restrictions on the matrices occurring in (\ref{e401}) and (\ref{e402}) is that they be conformable for multiplication).

\begin{corollary}\label{c401}
	For all $X$,
	\begin{equation}\label{e403}
		\|AX-B\|_F\geq\|AA^+B-B\|_F
	\end{equation}
	with the equality occurring in (\ref{e403}) if and only if $X=A^+B + (I_n-A^+A)L$,
	where $L\in R^{n\times p}$ is arbitrary; and
	\begin{equation}\label{e404}
		\|A^+B+(I_n-A^+A)L\|_F\geq\|A^+B\|_F
	\end{equation}
	with the equality occurring in (\ref{e404}) if and only if $(I_n-A^+A)L=0$.
\end{corollary}
\begin{proof}
	When $p=1$, (\ref{e403}) and (\ref{e404}) coincide with (\ref{e401}) and (\ref{e402}) respectively. When $p>1$, we have
	$$\|AX-B\|_F^2=\sum\limits_{j=1}^p\|AX_{:,j}-B_{:,j}\|_2^2\geq\sum\limits_{i=1}^p\|AA^+B_{:,j}-B_{:,j}\|_2^2=\|AA^+B-B\|_F^2,$$
	and
	$$\|A^+B+(I_n-A^+A)L\|_F^2=\sum\limits_{i=1}^p\|A^+B_{:,j}+(I_n-A^+A)L_{:,j}\|_2^2\geq\sum\limits_{i=1}^p\|A^+B_{:,j}\|_2^2=\|A^+B\|_F^2.$$
\end{proof}

\subsection{Solving Matrix Equation by Extended Kaczmarz Method}
Based Corollary \ref{c401}, we can solve matrix equation $AX=B$ or $XA=C$ for arbitrary $A\in R^{m\times n}$ by extended Kaczmarz method, where $A$ is likely to be rank defective ({\bf but there are no rows and columns that are all zero!}).

For $AX=B$, we use the following algorithm \ref{alg401}.

\begin{algorithm}
	\leftline{\caption{REK Method for Inconsistent Matrix Equation $AX=B$ (REKIAX)\label{alg401}}}
	\begin{algorithmic}[1]
		\Require
		$A\in R^{m\times n}$, $B\in R^{m\times p}$, $X^{(0)}\in R^{n\times p}$, $K\in R$,$Z_0=B$
		\State for $i=1:m$, $M(i)=\|A_{i,:}\|_2^2$
		\State for $j=1:n$, $N(j)=\|A_{:,j}\|_2^2$
		\For {$k=1,2,\cdots, K-1$}
		\State Set $p_{row=i}=\frac{\|A_{i,:}\|_2^2}{\|A \|^2_F}$, $p_{col=j}=\frac{\|A_{:,j}\|_2^2}{\|A \|^2_F}$
		\State Compute $Z^{(k+1)}=Z^{(k)}-\frac{A_{:,j}}{N(j)}((A_{:,j})^TZ^{(k)})$
		\State Compute $X^{(k+1)}=X^{(k)}+\frac{(A_{i,:})^T(B_{i,:}-Z_{i,:}-A_{i,:}X^{(k)})}{M(i)}$
		\EndFor
		\State Output $X^{(K)}$
	\end{algorithmic}
\end{algorithm}

\begin{lemma}\label{l401}
	Let $Z^*=(I_m- AA^+)B$. Denote $\{Z^{(k)}\}$ as the $k$th iterate of RK applied to $A^TZ=0$ with the initial guess $Z^{(0)}\in R^{m\times n}$. If $Z^{(0)}_{:,j}\in B_{:,j}+R(A),\ j=1,\ldots, n$,  then $Z^{(k)}$ converges linearly to $(I- AA^+)B$ in mean square form. Moreover, the solution error in expectation for the iteration sequence $Z^{(k)}$ obeys
	\begin{equation}\label{e405}
		E[\|Z^{(k)}-Z^*\|_F^2]\leq \rho^{k}\|Z^{(0)}-Z^*\|_F^2,\  {\rm with\ } \rho=1-\frac{\sigma^2_{\min}(A)}{\|A\|^2_F},
	\end{equation}
	where the $j$th column of $A$ is selected with probability $\hat{p}_j(A)=\frac{\|A_{:,j}\|_2^2}{\|A\|^2_F}$.
\end{lemma}

\begin{proof}
	It is easy to see that
	\begin{equation}\label{e406}
		\left\| Z^{(k)} -Z^*\right\|_F^2=\|Z^{(k)}-Z^{(k+1)}\|_F^2+\left\|Z^{(k+1)} - Z^*\right\|_F^2+2\langle Z^{(k)}-Z^{(k+1)}, Z^{(k+1)}-Z^*\rangle_F.
	\end{equation}
	It follows from
	\begin{align}\label{e407}
		E_k\left[\langle Z^{(k)}-Z^{(k+1)}, Z^{(k+1)}-Z^*\rangle_F\right] \nonumber
		=&	E_k\left[\left\langle \frac{A_{:,j}(A_{:,j})^TZ^{(k)}}{\|A_{:,j}\|_2^2}, Z^{(k+1)}-Z^*\right\rangle_F \right]\\ \nonumber
		=&	E_k\left[{\rm trace}\left(\frac{{Z^{(k)}}^TA_{:,j}(A_{:,j})^T}{\|A_{:,j}\|_2^2}(Z^{(k+1)}-Z^*)\right)\right]\\  \nonumber
		=&E_k\left[\frac{{Z^{(k)}}^TA_{:,j}(A_{:,j})^T}{\|A_{:,j}\|_2^2}Z^*\right] \ ({\rm by} \ (A_{:,j})^{T}Z^{(k+1)} = 0)\\  \nonumber
		=&{\rm trace}\left(\frac{{Z^{(k)}}^TAA^{T}(I-AA^+B)}{\|A\|_
			F^2}\right) \\
		=&0\ ({\rm by} \ A^{T}AA^+ =A^{T})
	\end{align}
	and
	\begin{align*}
		E_k\left[\|Z^{(k)}-Z^{(k+1)}\|_F^2\right]=& E_k\left[\left\|\frac{A_{:,j}(A_{:,j})^TZ^{(k)}}{\|A_{:,j}\|_2^2}\right\|_F^2\right]\\
		& =E_k\left[  {\rm trace}\left(\left(\frac{A_{:,j}(A_{:,j})^TZ^{(k)}}{\|A_{:,j}\|_2^2} \right)^T \frac{A_{:,j}(A_{:,j})^TZ^{(k)}}{\|A_{:,j}\|_2^2}  \right)\right]\\
		&={\rm trace}\left(\frac{{Z^{(k)}}^{T}AA^TZ^{(k)}}{\|A\|_F^2}\right)	\\	
		&=\left\|\frac{A^T(Z^{(k)}-Z^*)}{\|A\|_F}\right\|_F^2({\rm by} \ A^{T}(I-AA^+B) =0).
	\end{align*}
	Since $(Z^{(k)}-Z^*)_{:,j}\in R(A)$, we have
	\begin{align}\label{e408}
		E_k\left[\|Z^{(k)}-Z^{(k+1)}\|_F^2\right]\ge\frac{\sigma^2_{\min}(A)}{\|A\|_F^2}\left\| Z^{(k)} - Z^*\right\|_F^2
	\end{align}
	
	By taking the conditional expectation and Equations \eqref{e406},\eqref{e407},\eqref{e408}, we have
	\begin{align}\label{e409}
		E_k\left[\left\|Z^{(k+1)} - Z^*\right\|_F^2\right]=\left\| Z^{(k)} - Z^*\right\|_F^2-E_k\left[\|Z^{(k)}-Z^{(k+1)}\|_F^2\right]\nonumber
		\\
		\le \left ( 1- \frac{\sigma^2_{\min}(A)}{\|A\|_F^2}\right) \left\| Z^{(k)} - Z^*\right\|_F^2, \ k\ge 0.
	\end{align}
	Finally, by  (\ref{e409}) and induction on the iteration index $k$, we straightforwardly obtain the estimates (\ref{e405}). This completes the proof.
\end{proof}

\begin{theorem}\label{t401}
	The sequence $\{X^{(k)}\}$ generated by Algorithm \ref{alg401} starting from the initial matrix $X^{(0)}\in R^{n\times p}$, converges linearly to $A^+B$ in mean square if $X^{(0)}_{:,j}\in R(A^T),\ j=1,\cdots,p$, and
	\begin{equation*}\label{e410}
		E[\|X^{(k)}-A^+B\|_F^2]\leq\frac{k\rho^k}{ \|A\|^2_F}\left\|Z ^{(0)}-Z^*\right\|_F^2 + \rho^k \left\|X^{(0)}-A^+B \right\|_F^2,
	\end{equation*}
	where $\rho=1-\frac{\sigma_{min}^2(A)}{\|A\|^2_F}$, the $i$th row and $j$th column of $A$ are selected with probability $p_{row=i}=\frac{\|A_{i,:}\|_2^2}{\|A\|^2_F}$ and $p_{col=j}=\frac{\|A_{:,j}\|_2^2}{\|A\|^2_F}$, respectively.
\end{theorem}

\begin{proof}
	Denote ${X^{(k)}}$ as the $k$th iterate of REK method for $AX=B$, and $\tilde{X}^{(k+1)}  $  be   the one-step Kaczmarz update for the matrix equation $AX=AA^+B$ from $X^{(k)}$, i.e.,
	\begin{equation*}
		\tilde{X}^{(k+1)} =X^{(k)} + \frac{(A_{i,:})^T}{\|A_{i,:}\|^2_2} (A_{i,:}A^+B-A_{i,:}X^{(k)}).
	\end{equation*}
	We have
	\begin{align*}
		\tilde{X}^{(k+1)} -A^+B  & =X^{(k)}-A^+B  + \frac{(A_{i,:})^TA_{i,:}}{\|A_{i,:}\|^2_2} (A^+B- X^{(k)})  =  \left(I-\frac{(A_{i,:})^TA_{i,:}}{\|A_{i,:}\|^2_2}\right)(X^{(k)}-A^+B)
	\end{align*}
	and
	\begin{equation*}
		X^{(k+1)} - \tilde{X}^{(k+1)} = \frac{(A_{i,:})^T}{\|A_{i,:}\|^2_2} (B_{i,:}-Z_{i,:}^{(k+1)}-A_{i,:}A^+B).
	\end{equation*}
	It follows from
	\begin{align*}
		& \langle \tilde{X}^{(k+1)} -A^+B , X^{(k+1)} - \tilde{X}^{(k+1)} \rangle_F\\
		& = \left\langle  \left(I-\frac{(A_{i,:})^TA_{i,:}}{\|A_{i,:}\|^2_2}\right)(X^{(k)}-A^+B) , \  \frac{(A_{i,:})^T}{\|A_{i,:}\|^2_2} (B_{i,:}-Z_{i,:}^{(k+1)}-A_{i,:}A^+B) \right\rangle_F \\
		& ={\rm trace}\left((X^{(k)}-A^+B)^T \left(I-\frac{(A_{i,:})^TA_{i,:}}{\|A_{i,:}\|^2_2}\right) \frac{(A_{i,:})^T}{\|A_{i,:}\|^2_2} (B_{i,:}-Z_{i,:}^{(k+1)}-A_{i,:}A^+B) \right)\\
		&=0 \ ( {\rm by \ } \left(I-\frac{(A_{i,:})^TA_{i,:}}{\|A_{i,:}\|^2_2}\right) \frac{(A_{i,:})^T}{\|A_{i,:}\|^2_2}=0),
	\end{align*}
	and
	\begin{align*}
		\left\|X^{(k+1)} - \tilde{X}^{(k+1)}\right\|_F^2 & = \left\|\frac{(A_{i,:})^T}{\|A_{i,:}\|^2_2} (B_{i,:}-Z_{i,:}^{(k+1)}-A_{i,:}A^+B)\right\|_F^2 \\
		& ={\rm trace}\left((B_{i,:}-Z_{i,:}^{(k+1)}-A_{i,:}A^+B)^T \frac{A _{i,:} }{\|A_{i,:}\|^2_2}  \frac{(A_{i,:})^T}{\|A_{i,:}\|^2_2} (B_{i,:}-Z_{i,:}^{(k+1)}-A_{i,:}A^+B) \right)\\
		&=\frac{\left\|B_{i,:}-Z_{i,:}^{(k+1)}-A_{i,:}A^+B \right\|_2^2}{\|A_{i,:}\|^2_2}.
	\end{align*}
	that
	\begin{align}\label{e411}
		\left\|X^{(k+1)} - A^+B\right\|_F^2 &  = \left\|X^{(k+1)} - \tilde{X}^{(k+1)}\right\|_F^2 +\| \tilde{X}^{(k+1)} -A^+B \|_F^2 \notag\\
		& = \frac{\left\|B_{i,:}-Z_{i,:}^{(k+1)}-A_{i,:}A^+B \right\|_2^2}{\|A_{i,:}\|^2_2} +\| \tilde{X}^{(k+1)} -A^+B \|_F^2.
	\end{align}
	By taking the conditional expectation on the both side of (\ref{e411}), we have
	\begin{align*}
		E_{k} \left[\frac{\left\|B_{i,:}-Z_{i,:}^{(k+1)}-A_{i,:}A^+B \right\|_2^2}{\|A_{i,:}\|^2_2} \right]
		& = E_{k}^j E_{k}^i\left[\frac{\left\|B_{i,:}-Z_{i,:}^{(k+1)}-A_{i,:}A^+B \right\|_2^2}{\|A_{i,:}\|^2_2} \right]\\
		& = E_{k}^j\left[ \frac{1}{ \|A\|^2_F}\sum_{i=1}^m \left\|B_{i,:}-Z_{i,:}^{(k+1)}-A_{i,:}A^+B \right\|_2^2 \right]\\
		& =  \frac{1}{ \|A\|^2_F}E_{k}^j\left[ \left\|B -Z ^{(k+1)}-A A^+B \right\|_F^2  \right]\\
		& =  \frac{1}{ \|A\|^2_F}E_{k}\left[ \left\| Z ^{(k+1)}-(I-A A^+)B \right\|_F^2  \right],
	\end{align*}
	then
	\begin{align}\label{e412}
		E \left[\frac{\left\|B_{i,:}-Z_{i,:}^{(k+1)}-A_{i,:}A^+B \right\|_2^2}{\|A_{i,:}\|^2_2} \right]& = \frac{1}{ \|A\|^2_F}E \left[  \left\| Z ^{(k+1)}-Z^* \right\|_F^2   \right] \notag\\
		& \le \frac{\rho^{k+1}}{ \|A\|^2_F}\left\|Z ^{(0)}-Z^*\right\|_F^2 \ ({\rm by\ Lemma}\ \ref{l401}).
	\end{align}
	If $X^{(0)}_{:,j}\in R(A^T)$ and $(A^+B)_{:,j}\in R(A^T)$, $j=1,\ldots, n$, then $(X^{(k)}-A^+B)_{:,j}\in R(A^T),\ j=1,\ldots, n$ by induction. It follows from
	\begin{align*}
		E_{k}[ \|\tilde{X}^{(k+1)} -A^+B \|_F^2 ] & = E_{k}^i \left[\left\|\left(I-\frac{(A_{i,:})^TA_{i,:}}{\|A_{i,:}\|^2_2}\right)(X^{(k)}-A^+B)\right\|_F^2 \right]\\
		& = \sum_{i=1}^m  \frac{\|A_{i,:}\|^2_2}{ \|A\|^2_F}\left\|\left(I-\frac{(A_{i,:})^TA_{i,:}}{\|A_{i,:}\|^2_2}\right)(X^{(k)}-A^+B)\right\|_F^2\\
		& = \left\|X^{(k)}-A^+B \right\|_F^2-\frac{\|A(X^{(k)}-A^+B)\|_F^2}{\|A\|^2_F}  \\
		&\le \left\|X^{(k)}-A^+B \right\|_F^2- \frac{\sigma^2_{\min}(A)}{\|A\|^2_F}\left\|X^{(k)}-A^+B \right\|_F^2 \ ({\rm \ by\  Lemma}\ \ref{l201})\\
		&= \rho \left\|X^{(k)}-A^+B \right\|_F^2.
	\end{align*}
	that
	\begin{align}\label{e413}
		E \left[ \left\|\tilde{X}^{(k+1)} -A^+B \right\|_F^2 \right]\le \rho  E \left[ \left\|X^{(k)}-A^+B \right\|_F^2\right].
	\end{align}
	Combining (\ref{e411}),(\ref{e412}) and (\ref{e413}) yields
	\begin{align*}
		E \left[ \| X^{(k+1)} - A^+B\|_F^2 \right] & = E \left[\frac{\left\|B_{i,:}-Z_{i,:}^{(k+1)}-A_{i,:}A^+B \right\|_2^2}{\|A_{i,:}\|^2_2} \right] + E [ \|\tilde{X}^{(k+1)} -A^+B \|_F^2 ]\\
		&  \le \frac{\rho^{k+1}}{ \|A\|^2_F}\left\|Z ^{(0)}-Z^* \right\|_F^2 + \rho E \left[ \left\|X^{(k)}-A^+C \right\|_F^2 \right] \\
		& \le \frac{2\rho^{k+1}}{ \|A\|^2_F}\left\|Z ^{(0)}-Z^* \right\|_F^2 + \rho^2 E \left[ \left\|X^{(k-1)}-A^+B \right\|_F^2 \right] \\
		& \le \cdots \le \frac{(k+1)\rho^{k+1}}{ \|A\|^2_F}\left\|Z ^{(0)}-Z^* \right\|_F^2 + \rho^{k+1} \left\|X^{(0)}-A^+B \right\|_F^2 .
	\end{align*}
	This completes the proof.
\end{proof}

For the matrix equation $XA=C$, similar to Algorithm \ref{alg401}, we have the following algorithm \ref{alg402}.

\begin{algorithm}
	\leftline{\caption{REK Method for Inconsistent Matrix Equation $XA=C$ (REKIXA)}\label{alg402}}
	\begin{algorithmic}[1]
		\Require
		$A\in R^{m\times n}$, $C\in R^{p\times n}$, $X^{(0)}\in R^{p\times m}$, $K\in R$
		\State for $i=1:m$, $M(i)=\|A_{i,:}\|_2^2$
		\State for $j=1:n$, $N(j)=\|A_{:,j}\|_2^2$
		\For {$k=1,2,\cdots, K-1$}
		\State Set $p_{row=i}=\frac{\|A_{i,:}\|_2^2}{\|A \|^2_F}$, $p_{col=j}=\frac{\|A_{:,j}\|_2^2}{\|A \|^2_F}$, $Z^{(0)}=C^T$
		\State Compute $Z^{(k+1)}=Z^{(k)}-\frac{(A_{i,:})^T}{M(i)}(A_{i,:}Z^{(k)})$
		\State Compute $X^{(k+1)}=X^{(k)}+\frac{C_{:,j}-Z_{j,:}^T-X^{(k)}A_{:,j}}{N(j)}(A_{:,j})^T$
		\EndFor
		\State Output $X^{(K)}$
	\end{algorithmic}
\end{algorithm}

\begin{theorem}\label{t402}
	The sequence $\{X^{(k)}\}$ generated by Algorithm \ref{alg402} starting from the initial matrix $X^{(0)}\in R^{p\times m}$, converges linearly to $CA^+$ in mean square form if all the rows of $(X^{(0)}_{i,:})^T\in R(A),\ i=1,\cdots,m$, and
	\begin{equation*}\label{e414}
		E[\|X^{(k)}-CA^+\|_F^2]\leq\frac{k\rho^k}{\|A\|_F^2}\|Z^{(0)}-C(I_n-A^+A)\|_F^2+\rho^k\|X^{(0)}-CA^+\|_F^2,
	\end{equation*}
	where $\rho=1-\frac{\sigma_{min}^2(A)}{\|A\|^2_F}$, the $i$th row and $j$th column of $A$ are selected with probability $p_{row=i}=\frac{\|A_{i,:}\|_2^2}{\|A\|^2_F}$ and $p_{col=j}=\frac{\|A_{:,j}\|_2^2}{\|A\|^2_F}$, respectively.
\end{theorem}

The proof is similar to the proof of Theorem \ref{alg401}.

\begin{remark}\label{r401}
	Especially, let $p=m$ and $B=I_m$, or $p=n$ and $C=I_n$, that is, consider $AX=I_m$, or $XA=I_n$,  we can solve MP inverse of arbitrary $A\in R^{m\times n}$ by extended Kaczmarz method.
	Based on Theorem \ref{t401} ($B=I_m$) and Theorem \ref{t402} ($C=I_n$), we obtain the following result for arbitrary matrix $A\in R^{m\times n}$ (maybe rank defective)
	\begin{equation*}\label{415}
		\lim\limits_{k\rightarrow\infty}X^{(k)}=A^+=\arg\min\left\{\|X\|_F: \ X\in\arg\min\limits_{X\in R^{n\times m}}\|AX-I_m\|_F\right\},
	\end{equation*}
	and
	\begin{equation*}\label{416}
		\lim\limits_{k\rightarrow\infty}X^{(k)}=A^+=\arg\min\left\{\|X\|_F: \ X\in\arg\min\limits_{X\in R^{n\times m}}\|XA-I_n\|_F\right\}.
	\end{equation*}
\end{remark}

\subsection{Solving Matrix Equation by REGS Method}
Similar to Algorithm \ref{alg401} and Algorithm \ref{alg402}, we can solve matrix equation $AX=B$ and $XA=C$ for arbitrary $A\in R^{m\times n}$ by extended RGS method, where $A$ is likely to be rank defective ({\bf but there are no rows and columns that are all zero!}).
\begin{algorithm}
	\leftline{\caption{REGS Method for Inconsistent Matrix Equation $AX=B$ (REGSIAX)}\label{alg403}}
	\begin{algorithmic}[1]
		\Require
		$A\in R^{m\times n}$, $B\in R^{m\times p}$, $Y^{(0)}\in R^{n\times p}$, $X^{(0)}_{:,j}\in R(A^T)$, $j=1,\cdots,p$, $R^{(0)}=B-AY^{(0)}$, $K\in R$
		\State for $i=1:m$, $M(i)=\|A_{i,:}\|_2^2$
		\State for $j=1:n$, $N(j)=\|A_{:,j}\|_2^2$
		\For {$k=0,1,2,\cdots, K-1$}
		\State Set $p_{row=i}=\frac{\|A_{i,:}\|_2^2}{\|A \|^2_F}$, $p_{col=j}=\frac{\|A_{:,j}\|_2^2}{\|A \|^2_F}$
		\State Compute  $W^{(k)}=\frac{(A_{:,j})^T R^{(k)}}{N(j)}$, $Y^{(k+1)}_{j,:}=Y^{(k)}_{j,:}+W^{(k)}$, $R^{(k+1)}=R^{(k)}-A_{:,j}W^{(k)}$
		\State Compute $X^{(k+1)}=X^{(k)}-(A_{i,:})^T\left(\frac{A_{i,:}(X^{(k)}-Y^{(k+1)})}{M(i)}\right)$
		\EndFor
		\State Output $X^{(K)}$
	\end{algorithmic}
\end{algorithm}
\begin{theorem}\label{t403}
	The sequence $\{X^{(k)}\}$ generated by Algorithm \ref{alg403} starting from the initial matrix $X^{(0)}\in R^{n\times p}$, converges linearly to $A^+B$ in mean square if $X^{(0)}_{:,j}\in R(A^T),\ j=1,\cdots,p$, and
	\begin{equation*}\label{e417}
		E[\|X^{(k+1)}-A^+B\|_F^2]\leq\frac{(k+1)\rho^{k+1}}{ \|A\|^2_F}\left\|AY ^{(0)}-AA^+B \right\|_F^2 + \rho^{k+1} \left\|X^{(0)}-A^+B \right\|_F^2,
	\end{equation*}
	where $\rho=1-\frac{\sigma_{min}^2(A)}{\|A\|^2_F}$, the $i$th row and $j$th column of $A$ are selected with probability $p_{row=i}=\frac{\|A_{i,:}\|_2^2}{\|A\|^2_F}$ and $p_{col=j}=\frac{\|A_{:,j}\|_2^2}{\|A\|^2_F}$, respectively.
\end{theorem}
\begin{proof}
	Let ${Y^{(k+1)}}$ denote the $(k+1)$th iterate of RGS method for $AX=B$, and $X^{(k+1)} $ be the one-step Kaczmarz update for the matrix equation $AX=AY^{(k+1)}$ from $Y^{(k+1)}$(\cite{DK19}).
	\begin{align*}
		X^{(k+1)} -A^+B =  \left(I-\frac{(A_{i,:})^TA_{i,:}}{\|A_{i,:}\|^2_2}\right)(X^{(k)}-A^+B)+\frac{(A_{i,:})^TA_{i,:}}{\|A_{i,:}\|^2_2}(Y^{(k+1)}-A^+B)
	\end{align*}
	It follows from the orthogonality, namely,
	\begin{align*}
		& \left\langle  \left(I-\frac{(A_{i,:})^TA_{i,:}}{\|A_{i,:}\|^2_2}\right)(X^{(k)}-A^+B), \frac{(A_{i,:})^TA_{i,:}}{\|A_{i,:}\|^2_2}(Y^{(k+1)}-A^+B) \right\rangle_F\\
		& ={\rm trace}\left((X^{(k)}-A^+B)^T \left(I-\frac{(A_{i,:})^TA_{i,:}}{\|A_{i,:}\|^2_2}\right) \frac{(A_{i,:})^T}{\|A_{i,:}\|^2_2} (Y^{(k+1)}-A^+B) \right)\\
		&=0 \ ( {\rm by \ } \left(I-\frac{(A_{i,:})^TA_{i,:}}{\|A_{i,:}\|^2_2}\right) \frac{(A_{i,:})^T}{\|A_{i,:}\|^2_2}=0).
	\end{align*}
	Then
	\begin{align}\label{e418}
		\left\|X^{(k+1)} - A^+B\right\|_F^2 &  = \left\|\left(I-\frac{(A_{i,:})^TA_{i,:}}{\|A_{i,:}\|^2_2}\right)(X^{(k)}-A^+B)\right\|_F^2 +\left\| \frac{(A_{i,:})^TA_{i,:}}{\|A_{i,:}\|^2_2}(Y^{(k+1)}-A^+B) \right\|_F^2.
	\end{align}
	By taking the conditional expectation on the both side of (\ref{e418}), we have
	\begin{align*}
		& E_{k} \left[\left\| \frac{(A_{i,:})^TA_{i,:}}{\|A_{i,:}\|^2_2}(Y^{(k+1)}-A^+B) \right\|_F^2 \right]\\
		&=	E_{k} \left[{\rm trace}\left((Y^{(k+1)}-A^+B)^T \frac{(A_{i,:})^TA_{i,:} }{\|A_{i,:}\|^2_2}  \frac{(A_{i,:})^TA_{i,:}}{\|A_{i,:}\|^2_2} (Y^{(k+1)}-A^+B) \right) \right]\\
		&=E_{k}^j E_{k}^i \left[{\rm trace}\left((Y^{(k+1)}-A^+B)^T \frac{(A_{i,:})^TA_{i,:}}{\|A_{i,:}\|^2_2} (Y^{(k+1)}-A^+B) \right) \right]\\
		& = E_{k}^j\left[ \frac{1}{ \|A\|^2_F}\sum_{i=1}^m \left\|A_{i,:}(Y^{(k+1)}-A^+B) \right\|_2^2 \right]\\
		& =  \frac{1}{ \|A\|^2_F}E_{k}^j\left[ \left\|A(Y^{(k+1)}-A^+B) \right\|_F^2  \right]\\
		& =  \frac{1}{ \|A\|^2_F}E_{k}\left[ \left\| A(Y^{(k+1)}-A^+B) \right\|_F^2  \right],
	\end{align*}
	then
	\begin{align}\label{e419}
		E_{k} \left[\left\| \frac{(A_{i,:})^TA_{i,:}}{\|A_{i,:}\|^2_2}(Y^{(k+1)}-A^+B) \right\|_F^2 \right]& = \frac{1}{ \|A\|^2_F}E_{k}\left[ \left\| A(Y^{(k+1)}-A^+B) \right\|_F^2  \right] \notag\\
		& \le \frac{\rho^{k+1}}{ \|A\|^2_F}\left\|A(Y^{(0)}-A^+B)\right\|_F^2 \ ({\rm \ by\  Theorem}\ \ref{t301})
	\end{align}
	If $X^{(0)}_{:,j}\in R(A^T)$ and $(A^+B)_{:,j}\in R(A^T)$, $j=1,\ldots, n$, then $(X^{(k)}-A^+B)_{:,j}\in R(A^T),\ j=1,\ldots, n$ by induction. It follows from
	\begin{align}\label{e420}
		E_{k}\left[\left\|\left(I-\frac{(A_{i,:})^TA_{i,:}}{\|A_{i,:}\|^2_2}\right)(X^{(k)}-A^+B)\right\|_F^2 \right] & = E_{k}^i \left[\left\|\left(I-\frac{(A_{i,:})^TA_{i,:}}{\|A_{i,:}\|^2_2}\right)(X^{(k)}-A^+B)\right\|_F^2 \right] \notag\\
		& = \sum_{i=1}^m  \frac{\|A_{i,:}\|^2_2}{ \|A\|^2_F}\left\|\left(I-\frac{(A_{i,:})^TA_{i,:}}{\|A_{i,:}\|^2_2}\right)(X^{(k)}-A^+B)\right\|_F^2 \notag\\
		& = \left\|X^{(k)}-A^+B \right\|_F^2-\frac{\|A(X^{(k)}-A^+B)\|_F^2}{\|A\|^2_F} \notag \\
		&\le \left\|X^{(k)}-A^+B \right\|_F^2- \frac{\sigma^2_{\min}(A)}{\|A\|^2_F}\left\|X^{(k)}-A^+B \right\|_F^2 \notag\\
		&= \rho \left\|X^{(k)}-A^+B \right\|_F^2,
	\end{align}
	where the inequality is based on Lemma \ref{l201}.	
	Combining (\ref{e418}),(\ref{e419}) and (\ref{e420}) it yields
	\begin{align*}
		&E \left[ \| X^{(k+1)} - A^+B\|_F^2 \right]\\
		& = E \left[\left\|\left(I-\frac{(A_{i,:})^TA_{i,:}}{\|A_{i,:}\|^2_2}\right)(X^{(k)}-A^+B)\right\|_F^2 \right] + E\left[ \left\| \frac{(A_{i,:})^TA_{i,:}}{\|A_{i,:}\|^2_2}(Y^{(k+1)}-A^+B) \right\|_F^2\right]\\
		&  \le \frac{\rho^{k+1}}{ \|A\|^2_F}\left\|A(Y^{(0)}-A^+B) \right\|_F^2 + \rho E \left[ \left\|X^{(k)}-A^+C \right\|_F^2 \right] \\
		& \le \frac{2\rho^{k+1}}{ \|A\|^2_F}\left\|A(Y^{(0)}-A^+B)\right\|_F^2 + \rho^2 E \left[ \left\|X^{(k-1)}-A^+B \right\|_F^2 \right] \\
		& \le \cdots \le \frac{(k+1)\rho^{k+1}}{ \|A\|^2_F}\left\|A(Y^{(0)}-A^+B) \right\|_F^2 + \rho^{k+1} \left\|X^{(0)}-A^+B \right\|_F^2 .
	\end{align*}
	This completes the proof.
\end{proof}

For the matrix equation $XA=C$, similar to Algorithm \ref{alg403}, we have the following algorithm \ref{alg404}.

\begin{theorem}\label{t404}
	The sequence $\{X^{(k)}\}$ generated by Algorithm \ref{alg404} starting from the initial matrix $X^{(0)}\in R^{p\times m}$, converges linearly to $CA^+$ in mean square if $(X^{(0)}_{i,:})^T\in R(A),\ j=1,\cdots,m$, and
	\begin{equation*}\label{e421}
		E[\|X^{(k+1)}-CA^+\|_F^2]\leq\frac{(k+1)\rho^{k+1}}{ \|A\|^2_F}\left\|Y ^{(0)}A-CA^+A \right\|_F^2 + \rho^{k+1} \left\|X^{(0)}-CA^+\right\|_F^2,
	\end{equation*}
	where $\rho=1-\frac{\sigma_{min}^2(A)}{\|A\|^2_F}$, the $i$th row and $j$th column of $A$ are selected with probability $p_{row=i}=\frac{\|A_{i,:}\|_2^2}{\|A\|^2_F}$ and $p_{col=j}=\frac{\|A_{:,j}\|_2^2}{\|A\|^2_F}$, respectively.
\end{theorem}

The proof is similar to the proof of Theorem \ref{alg403}.

\begin{algorithm}
	\leftline{\caption{REGS Method for Inconsistent Matrix Equation $XA=C$ (REGSIXA)}\label{alg404}}
	\begin{algorithmic}[1]
		\Require
		$A\in R^{m\times n}$, $C\in R^{p\times n}$, $Y^{(0)}\in R^{p\times m}$, $(X^{(0)}_{i,:})^T\in R(A)$, $j=1,\cdots,m$, $R^{(0)}=C-Y^{(0)}A$, $K\in R$
		\State for $i=1:m$, $M(i)=\|A_{i,:}\|_2^2$
		\State for $j=1:n$, $N(j)=\|A_{:,j}\|_2^2$
		\For {$k=0,1,2,\cdots, K-1$}
		\State Set $p_{row=i}=\frac{\|A_{i,:}\|_2^2}{\|A \|^2_F}$, $p_{col=j}=\frac{\|A_{:,j}\|_2^2}{\|A \|^2_F}$
		\State Compute $U^{(k)}=\frac{R^{(k)}(A_{i,:})^T}{\|A_{i,:}\|_2^2}, \ Y^{(k+1)}_{:,i}=Y^{(k)}_{:,i}+U^{(k)}, \  R^{(k+1)}=R^{(k)}-U^{(k)}A_{i,:}$
		\State Compute $X^{(k+1)}=X^{(k)}-\frac{(X^{(k)}-Y^{(k+1)})A_{:,j}}{N(j)}(A_{:,j})^T$
		\EndFor
		\State Output $X^{(K)}$
	\end{algorithmic}
\end{algorithm}

\begin{remark}\label{r402}
	Especially, let $p=m$ and $B=I_m$, or $p=n$ and $C=I_n$, that is, consider $AX=I_m$, or $XA=I_n$,  we can solve MP inverse of arbitrary $A\in R^{m\times n}$ by extended GS method.
	Based on Theorem \ref{t403} ($B=I_m$) and Theorem \ref{t404} ($C=I_n$), we obtain the following result for arbitrary matrix $A\in R^{m\times n}$ (maybe rank defective)
	\begin{equation*}\label{e422}
		\lim\limits_{k\rightarrow\infty}X^{(k)}=A^+=\arg\min\limits_{X=\arg\min\|X\|_F}\|AX-I_m\|_F,
	\end{equation*}
	and
	\begin{equation*}\label{e423}
		\lim\limits_{k\rightarrow\infty}X^{(k)}=A^+=\arg\min\limits_{X=\arg\min\|X\|_F}\|AX-I_n\|_F.
	\end{equation*}
\end{remark}

\section{Recursive Methods for Matrix Equation}
When the solution of the matrix equation is found, it is often encountered that new data is added. So can we use the existing results instead of starting from scratch?

If the results of an additional experiment or observation become available after (\ref{e201}) is solved ($A^+B$ is the minimal $F$-norm solution or least square solution), then it is necessary to update the solution of the following equation (\ref{e501}) in light of the additional information,
\begin{equation}\label{e501}
	\tilde{A}X=\tilde{B},
\end{equation}
where $\tilde{A}=[A^T,a^T]^T$, $\tilde{B}=[B^T,b^T]^T$, $X\in R^{n\times p}$ and $a\in R^{1\times n}$, $b\in R^{1\times p}$. Based on \cite{Gre60}, we have
\begin{equation*}\label{e502}
	\tilde{A}^+=\left\{
	\begin{array}{lr}
		\left(A^+-(a-aA^+A)^+aA^+, \ \ (a-aA^+A)^+\right), & a-aA^+A\not=0,\\
		\left(A^+-\frac{(A^TA)^+a^TaA^+}{1+a(A^TA)^+a^T}, \ \ \frac{(A^TA)^+}{1+a(A^TA)^+a^T}\right), & a-aA^+A=0.
	\end{array}
	\right.
\end{equation*}
Therefore, the minimal $F$-norm solution or least square solution of Eq. (\ref{e501}) is
\begin{equation*}\label{e503}
	\tilde{X}=\tilde{A}^+\tilde{B}=A^+B+d(b-aA^+B),
\end{equation*}
where
\begin{equation*}\label{e504}
	d=\left\{
	\begin{array}{lr}
		(a-aA^+A)^+, & a-aA^+A\not=0,\\
		\frac{(A^TA)^+a^T}{1+a(A^TA)^+a^T}, & a-aA^+A=0.
	\end{array}
	\right.
\end{equation*}

In fact,
\begin{equation*}\label{e505}
	\bar{X}=A^+B+a^+(b-aA^+B)=A^+B+\frac{a^T}{\|a\|_2^2}(b-aA^+B)
\end{equation*}
is orthogonal projection of $A^+B$ onto the subspace $H=\left\{X\in R^{n\times p}: \ aX=b\right\}$. Therefore,

(1) If $aA^T=0$, i. e., data $a$ is orthogonal to all rows of $A$, $\tilde{X}=\bar{X}$;

(2) If $aA^+B\approx b$, we can use $\tilde{X}=CA^+$ as the initial iteration and continue to solve iteratively the new matrix equation (\ref{e501});

(3) If $m<<n$, we know that $I_n-A^+A\approx I_n$, the $\bar{X}$ can be used as the initial iteration and continue to solve iteratively the new matrix equation (\ref{e501}).

If new data is added to the column of the coefficient matrix $A$ after (\ref{e201}) is solved ($A^+B$ is the minimal $F$-norm solution or least square solution), we have similar methods to recursively solve new matrix equations
\begin{equation}\label{e506}
	X\tilde{A}=\tilde{C},
\end{equation}
where $\tilde{A}=[A,a_{n+1}]$, $\tilde{C}=[C,c_{n+1}]$, $a_{n+1}, \ c_{n+1}\in R^m$ and $X\in R^{p\times m}$. Based on \cite{Gre60}, we have
\begin{equation*}\label{e507}
	\tilde{A}^+=\left\{
	\begin{array}{lr}
		\left(\begin{array}{c}
			A^+-A^+a_{n+1}(a_{n+1}-AA^+a_{n+1})^+\\(a_{n+1}-AA^+a_{n+1})^+
		\end{array}\right), & a_{n+1}-AA^+a_{n+1}\not=0,\\
		\left(\begin{array}{c}
			A^+-\frac{A^+a_{n+1}a_{n+1}^T(AA^T)^+}{1+a_{n+1}^T(AA^T)^+a_{n+1}}\\ \frac{a_{n+1}^T(AA^T)^+}{1+a_{n+1}^T(AA^T)^+a_{n+1}}
		\end{array}\right), & a_{n+1}-AA^+a_{n+1}=0.
	\end{array}
	\right.
\end{equation*}
Therefore, the minimal $F$-norm solution or least square solution of Eq. (\ref{e506}) is
\begin{equation*}\label{e508}
	\tilde{X}=\tilde{C}\tilde{A}^+=CA^++(c_{n+1}-CA^+a_{n+1})d,
\end{equation*}
where
\begin{equation*}\label{e509}
	d=\left\{
	\begin{array}{lr}
		(a_{n+1}-AA^+a_{n+1})^+, & a_{n+1}-AA^+a_{n+1}\not=0,\\
		\frac{a_{n+1}^T(AA^T)^+}{1+a_{n+1}^T(AA^T)^+a_{n+1}}, &  a_{n+1}-AA^+a_{n+1}=0.
	\end{array}
	\right.
\end{equation*}
In fact,
\begin{equation*}\label{e510}
	\bar{X}=CA^++(c_{n+1}-CA^+a_{n+1})a_{n+1}^+
\end{equation*}
is orthogonal projection of $CA^+$ onto the subspace $H=\left\{X\in R^{n\times p}: \ Xa_{n+1}=c_{n+1}\right\}$. Therefore,

(1) If $A^Ta_{n+1}=0$, i. e., data $a_{n+1}$ is orthogonal to all columns of $A$, $\tilde{X}=\bar{X}$;

(2) If $CA^+a_{n+1}\approx c_{n+1}$, we can use $\tilde{X}=CA^+$ as the initial iteration and continue to solve iteratively the new matrix equation (\ref{e506});

(3) If $m>>n$, we know that $I_m-AA^+\approx I_m$, the $\bar{X}$ can be used as the initial iteration and continue to solve iteratively the new matrix equation (\ref{e506}).

\section{Numerical Experiments}
In this section, to verify the efficiency of the proposed algorithms, we will present some experiment results for matrix equations.
All experiments are carried out by using MATLAB (version R2020a) on a DESKTOP-8CBRR86
with Intel(R) Core(TM) i7-4712MQ CPU @2.30GHz   2.29GHz, RAM 8GB and Windows 10.

All computations are started from the initial guess $X^{(0)}=0, Y^{(0)}=0$, and terminated once the relative error (RE) of the solution, defined by
$$RE=\frac{\|X^{(k)}-X^*\|_F^2}{\|X^*\|_F^2}$$
at the the current iterate $X^{(k)}$, satisfies $RE<10^{-6}$ or exceeds maximum iteration $K=50000$, where $X^*$  represents the left inverse, the right inverse, $A^+$ or inverse for different matrices.
We report the average number of iterations (denoted as `IT') and the average computing time in second (denoted as `CPU') for 10 trials repeated runs of the corresponding methods.
In the following tables, the item `$>$' represents that the number of iteration steps exceeds the maximum iteration (50000), and the item `$-$' represents that the method does not converge.
We test the performance of various methods for the matrix equations $AX=B$ and $XA=C$ with synthetic dense  data and real-world sparse data.
\begin{itemize}
	\item Type I: For given $m,n$, the entries of $A$ is generated from a standard normal distribution, i.e., $A=randn(m,n).$ We also construct the rank-deficient matrix by $ A=randn(m,n/2), A=[A, A]$ or $ A=randn(m/2,n)$, $A=[A; A]$ and so on.
	\item Type II: The real-world sparse data come from the Florida sparse matrix collection \cite{DH11}. Table \ref{tab0} lists the  features of these sparse matrices.
\end{itemize}
\begin{table}[H]
	\caption{The detailed features of sparse matrices from \cite{DH11}.}
	\label{tab0}
	\centering
	\begin{tabular}{ c c c c  }
		\hline
		name & size  & rank   & sparsity     \\
		\hline
		ash219  & $219\times 85$   &  85  & $2.3529\%$  	\\
		\hline
		ash958 &  $958\times 292$  &  292  & $0.68493\%$  	\\
		\hline
		divorce &  $50\times 9$  &  9  & $50\%$  	\\
		\hline
		Worldcities &  $315\times 100$  &  100  & $53.625\%$  	\\
		\hline
	\end{tabular}
\end{table}

\subsection{Consistent Matrix Equation}
First, we compare the performance of the RK, REK, RGS and REGS methods for solving the consistent matrix equations $AX=B$ and $XA=C$. To construct two consistent matrix equations, we set $B=AX^*$ and $C=X^*A$, where $X^*$ is a random matrix which is generated by $X^*=randn(*,*)$.

\begin{example}\label{EX6.1}
	Random matrix. Synthetic dense data for this test is generated as follows: $A=randn(m,n), X^*_1=randn(n,p), X^*_2=randn(p,m), B=AX^*_1, C=X^*_2A$, and rank-deficient matrix $A=randn(m,n/2), A=[A, A]$. Numerical results are shown in Figure \ref{figure:6.1.1} and Table \ref{table:6.1.1}.
\end{example}

From Table \ref{table:6.1.1}, we can see that the RKCAX, RKCXA, RGSIAX and RGSIXA methods vastly outperform the REKIAX, REKIXA, REGSIAX and REGSIXA methods in terms of both IT and CPU times. The RKCAX and RKCXA methods have the least iteration steps and runs the least time regardless of whether the matrices A is full column/row rank or not. However, the RGSIAX and RGSIXA methods do not converge if the matrices A is not full column/row rank. As the increasing of matrix dimension, the CPU time of RKCAX, RKCXA, RGSIAX and RGSIXA is increasing slowly , while the running time of RKCAX, RKCXA, RGSIAX and RGSIXA increases dramatically.

Figure \ref{figure:6.1.1} shows the plots of relative error (RE) in base-10 logarithm versus IT and CPU of different methods with $AX=B$ ($m=50, n=30, p=30, A$ is full column rank) and $XA=C$ ($m=30, n=50, p=30, A$ is full row rank). We can see the relative errors of RKCAX, RKCXA, RGSIAX and RGSIXA are decreasing rapidly with the increase of iteration steps and the computing times.

\begin{example}\label{EX6.2}
	Real-world matrix. The entries of $A$ is selected from the real-world sparse data \cite{DH11}. Table \ref{tab0} lists the features of these sparse matrices. Let $X^*_1=randn(n,p), X^*_2=randn(p,m), B=AX^*_1, C=X^*_2A$.
	Numerical results are shown in Table \ref{table:6.2.1}.
\end{example}

For the sparse matrices from Type II, we list the numbers of iteration steps and the computing times for the RKCAX, RKCXA, RGSIAX, RGSIXA, REKIAX, REKIXA, REGSIAX, and REGSIXA methods in Table \ref{table:6.2.1}. For all cases in Table \ref{table:6.2.1}, the RKCAX, RKCXA, REKIAX, REKIXA, REGSIAX, and REGSIXA methods all converge to the solution, but the RKCAX and RKCXA methods are significantly better than the REKIAX, REKIXA, REGSIAX, and REGSIXA methods, both in terms of iteration steps and running time. For $A=$divorce, ash219, Worldcities, and ash958, the RGSIXA method does not converge because $A$ is not full row rank. For $A=$divorce$^{\top}$, ash219$^{\top}$, Worldcities$^{\top}$, and ash958$^{\top}$, the RRGSIAX method does not converge because $A$ is not full column rank.

\subsection{Inconsistent Matrix Equation}
Next, we compare the performance of the the RGSIAX, RGSIXA, REKIAX, REKIXA, REGSIAX, and REGSIXA methods for solving the inconsistent matrix equations $AX=B$ and $XA=C$. To construct two inconsistent matrix equations, we set $B=AX^*+R$ and $C=X^*A+R$, where $X^*$ and $R$ are random matrices which are generated by $X^*=randn(*, *)$ and $R=\delta\times randn(*, *), \delta \in \left(0, 1\right)$.

\begin{example}\label{EX6.3}
	Random matrix. Synthetic dense data for this test is generated as follows: $A=randn(m,n), X^*_1=randn(n,p), X^*_2=randn(p,m), B=AX^*_1+R_1, C=X^*_2A+R_2$ and $R_1=\delta\times randn(m,p), R_2=\delta \times randn(p,n)$ where $\delta = 10^{-5}$, and rank-deficient matrix $A=randn(m,n/2)$, $A=[A, A]$.
	Numerical results are shown in Figure \ref{figure:6.3.1} and Table \ref{table:6.3.1}.
\end{example}


In Table \ref{table:6.3.1}, we report the average IT and CPU of the RGSIAX, RGSIXA, REKIAX, REKIXA, REGSIAX, and REGSIXA methods for solving inconsistent matrix with Type I matrices.
We can see that the RGSIAX and RGSIXA methods are better than REKIAX, REKIXA, REGSIAX, and REGSIXA in terms of IT and CPU time. The fly in the ointment is that the convergence conditions of the RGSIAX and RGSIXA method are more stringent, which requires $A$ of $AX=B$ is full column rank and A of $XA=C$ is full row rank.  The REKIAX, REKIXA, REGSIAX, and REGSIXA methods can successfully solve the linear least-squares solution for all cases. Figure \ref{figure:6.3.1} shows the plots of relative error (RE) in base-10 logarithm versus IT and CPU of different methods with $AX=B$ ($m=50, n=30, p=30, A$ is full column rank) and $XA=C$ ($m=30, n=50, p=30, A$ is full row rank). Again, we can find the RGSIAX and RGSIXA methods converge faster than the REKIAX, REKIXA, REGSIAX, and REGSIXA methods.

\begin{example}\label{EX6.4}
	Real-world matrix. The entries of $A$ is selected from the real-world sparse data \cite{DH11}. Table \ref{tab0} lists the features of these sparse matrices. Let $X^*_1=randn(n,p), X^*_2=randn(p,m), p=10, B=AX^*_1+R_1, C=X^*_2A+R_2$ and $R_1=\delta \times randn(m,p), R_2=\delta\times randn(p,n)$ where $\delta = 10^{-5}$.
	Numerical results are shown in Table \ref{table:6.4.1}.
\end{example}

In Table \ref{table:6.4.1}, we list the average IT and CPU of the RGSIAX, RGSIXA, REKIAX, REKIXA, REGSIAX, and REGSIXA methods for solving inconsistent matrix with sparse matrices. We can observe that the RGSIAX and RGSIXA methods require less CPU and IT than REKIAX, REKIXA, REGSIAX, and REGSIXA methods.

\section{Conclusion}
We have proposed a series of Kaczmarz-type methods: RK method, REK method, RGS method and REGS method for solving the matrix $AX=B$ and $XA=C$. Same times, these methods can also be used to finding the right inverse, left inverse and Moore-Penrose generalized inverse of a matrix. These methods avoid calculating the product of matrix and matrix and are suitable for large-scale problems. The convergence of the random algorithms of these methods are also guaranteed. The numerical results show that these methods are very efficient and all these algorithms (RKCAX, RKCXA, RGSIAX, RGAIXA,  REKIAX, REKIXA, REGSIAX and REGSIXA) can be selected according to different situations.

\begin{figure}[htbp]
	\centering
	\includegraphics[width=4cm]{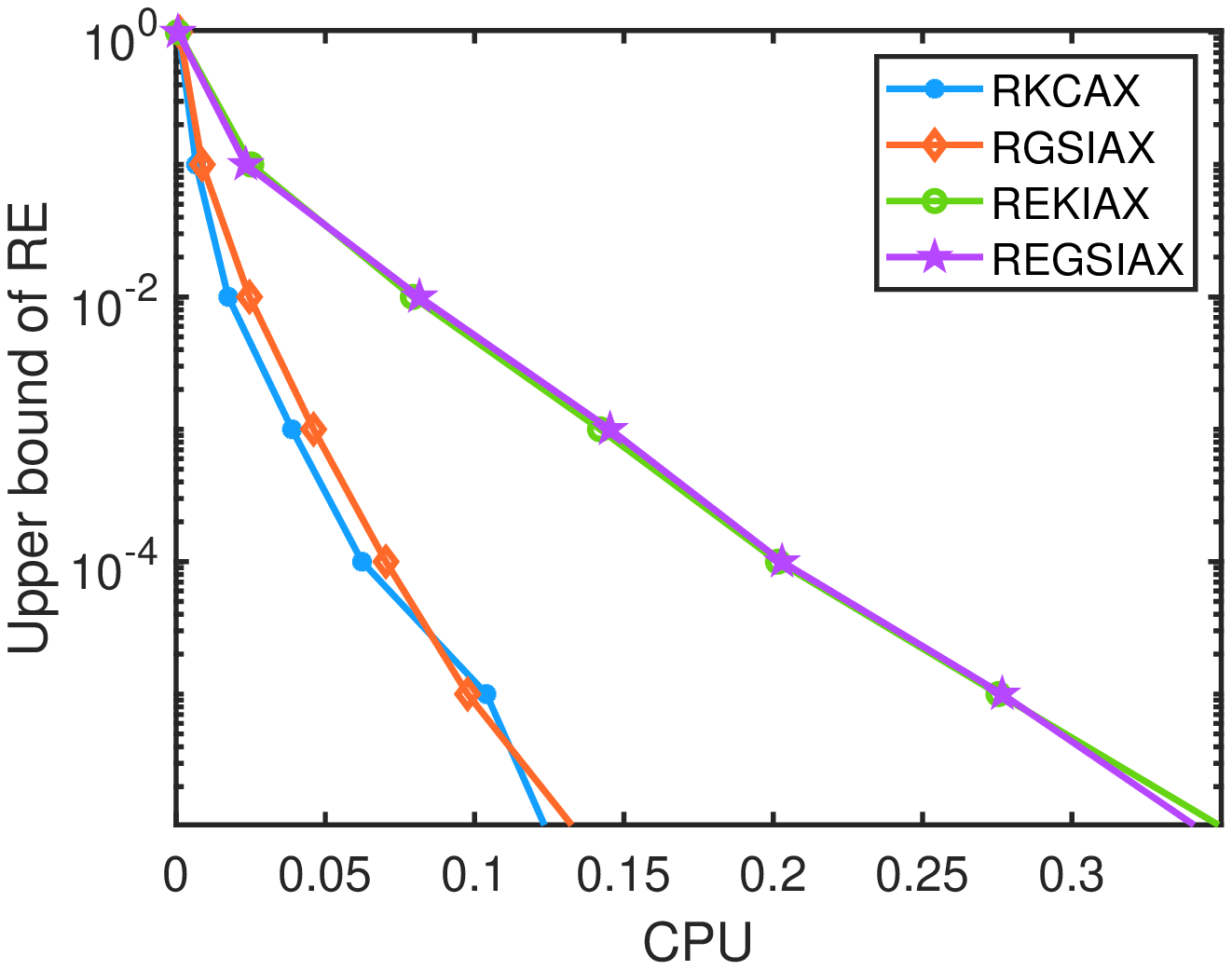}
	\includegraphics[width=4cm]{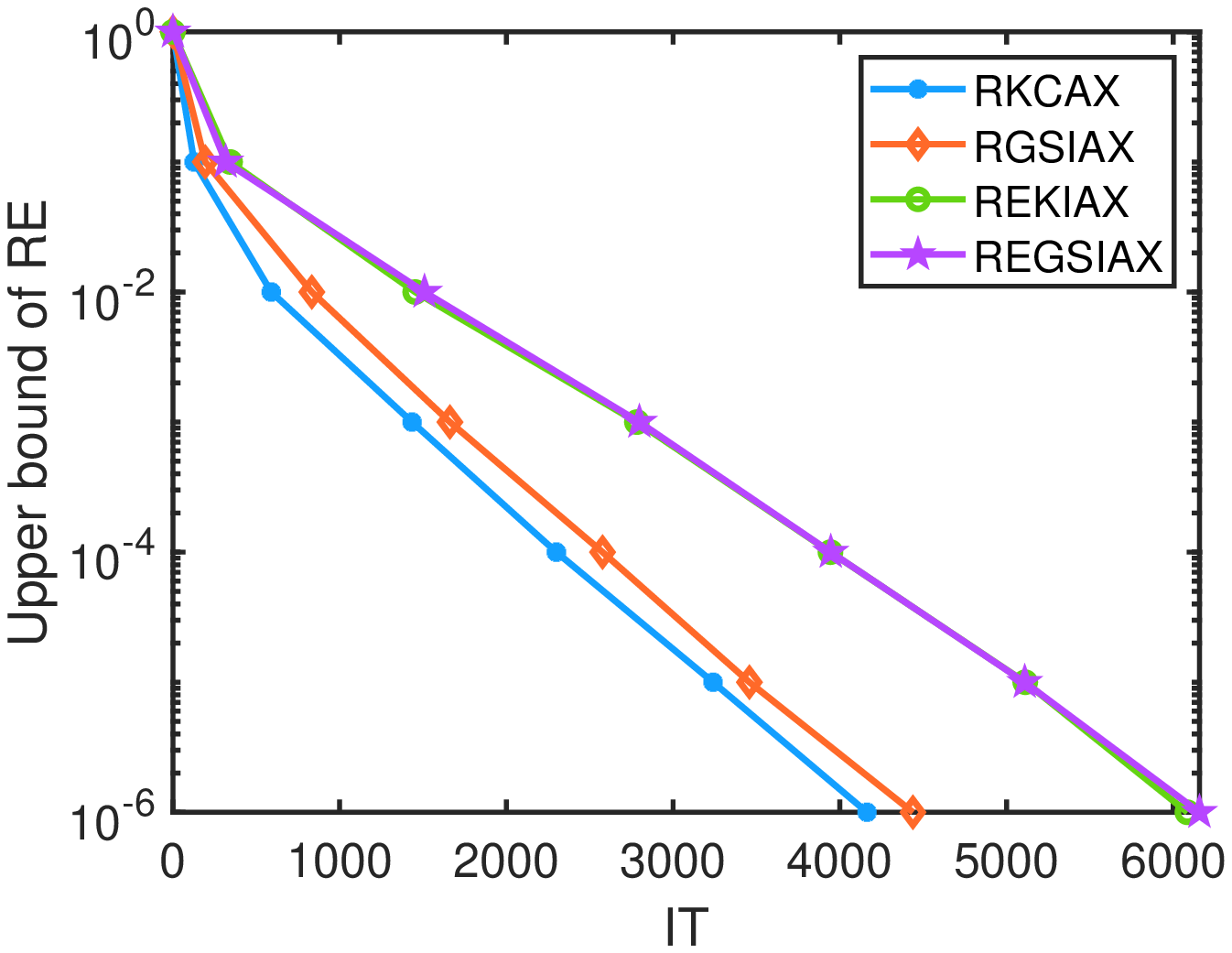}
	\includegraphics[width=4cm]{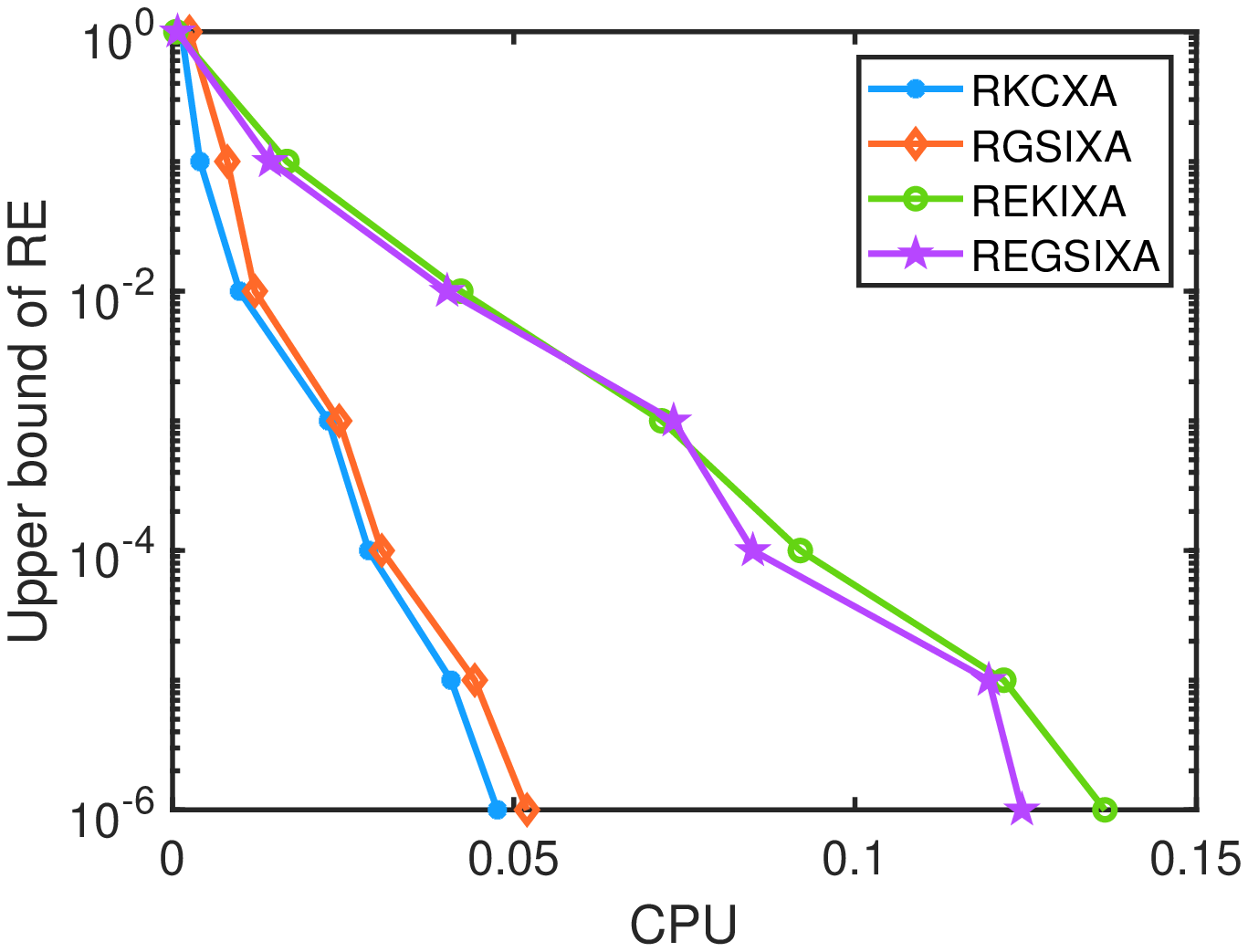}
	\includegraphics[width=4cm]{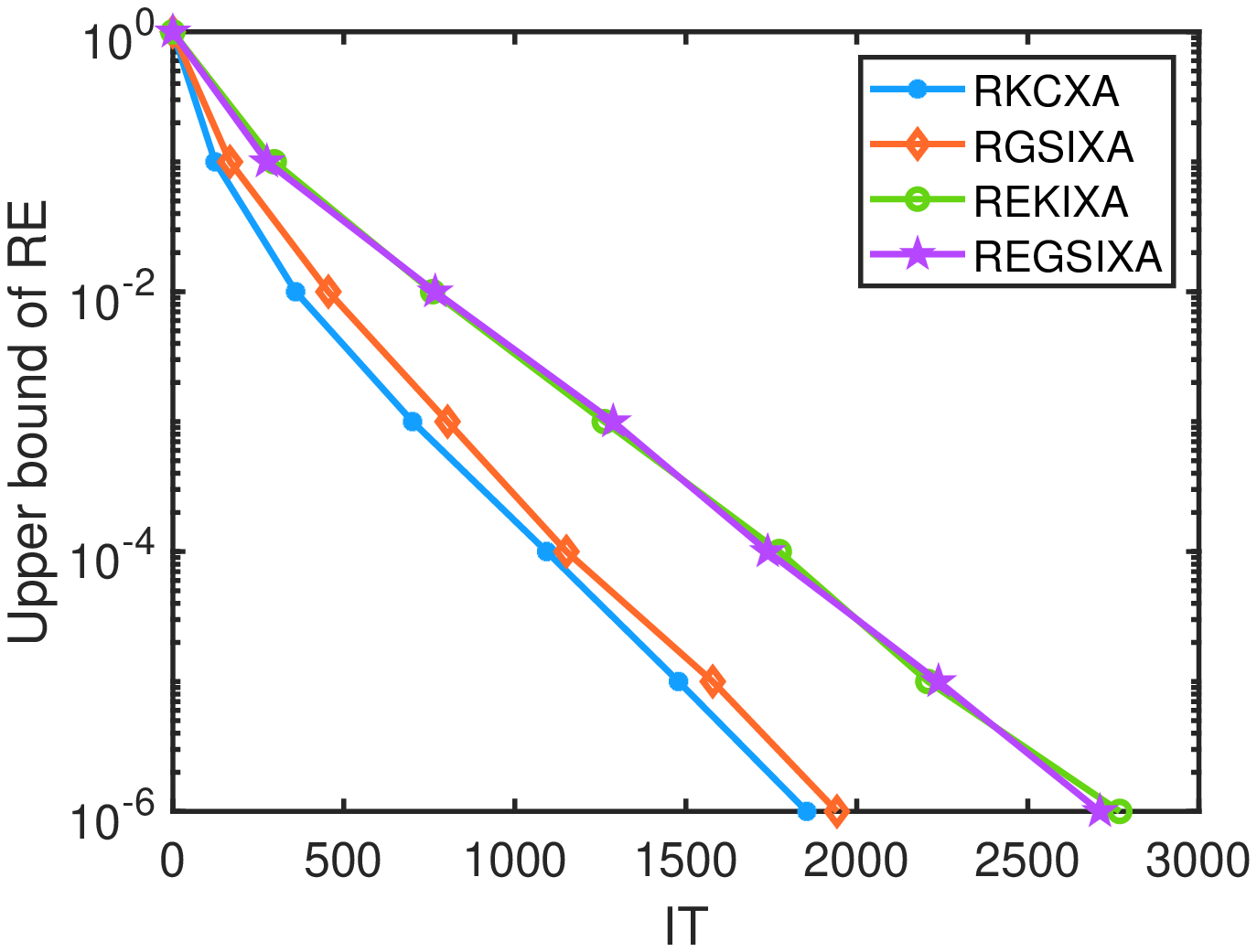}
	\caption{Relative errors of different methods for consistent matrix equations $AX=B$ and $XA=C$.}\label{figure:6.1.1}
\end{figure}

\begin{figure}[htbp]
	\centering
	\includegraphics[width=4cm]{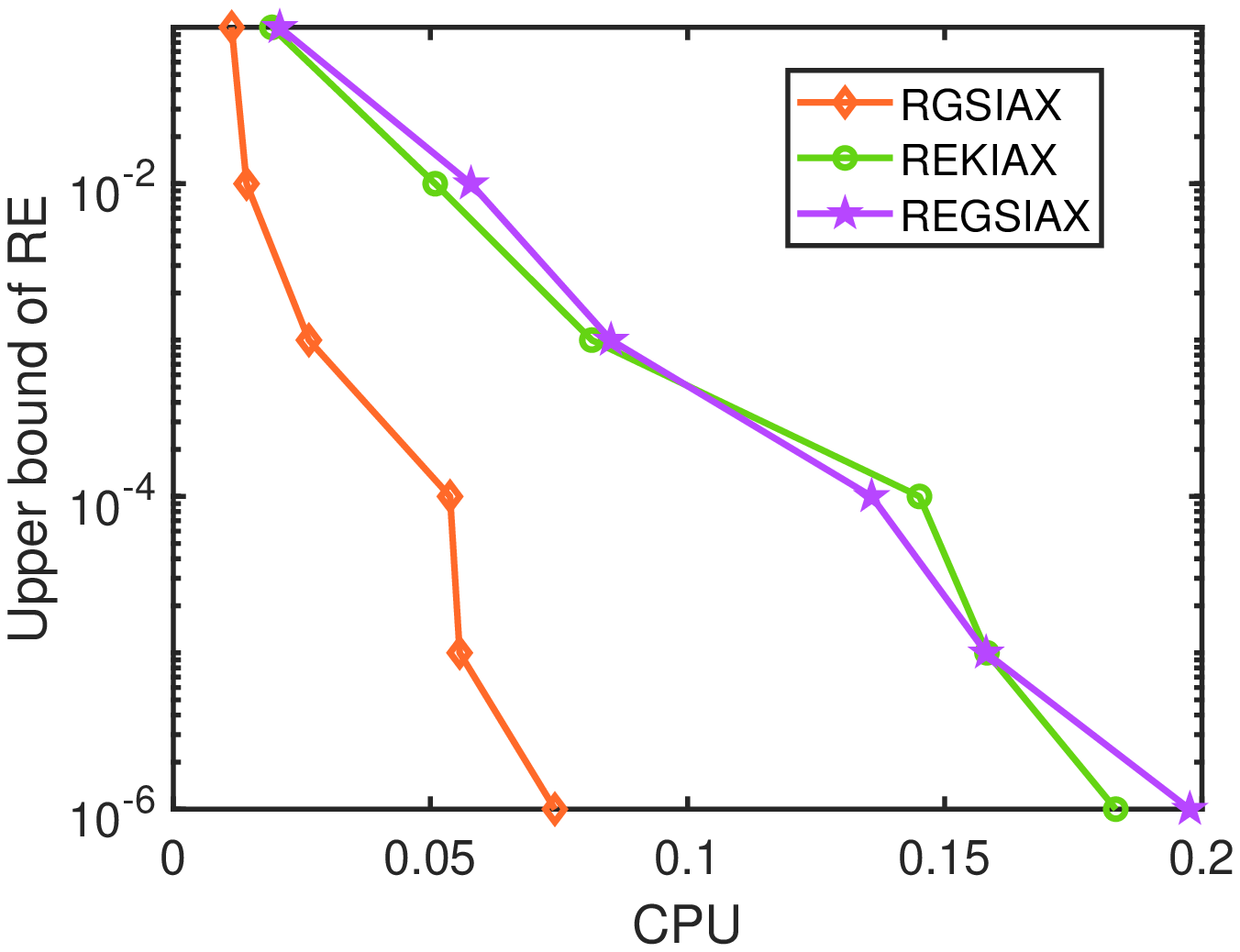}
	\includegraphics[width=4cm]{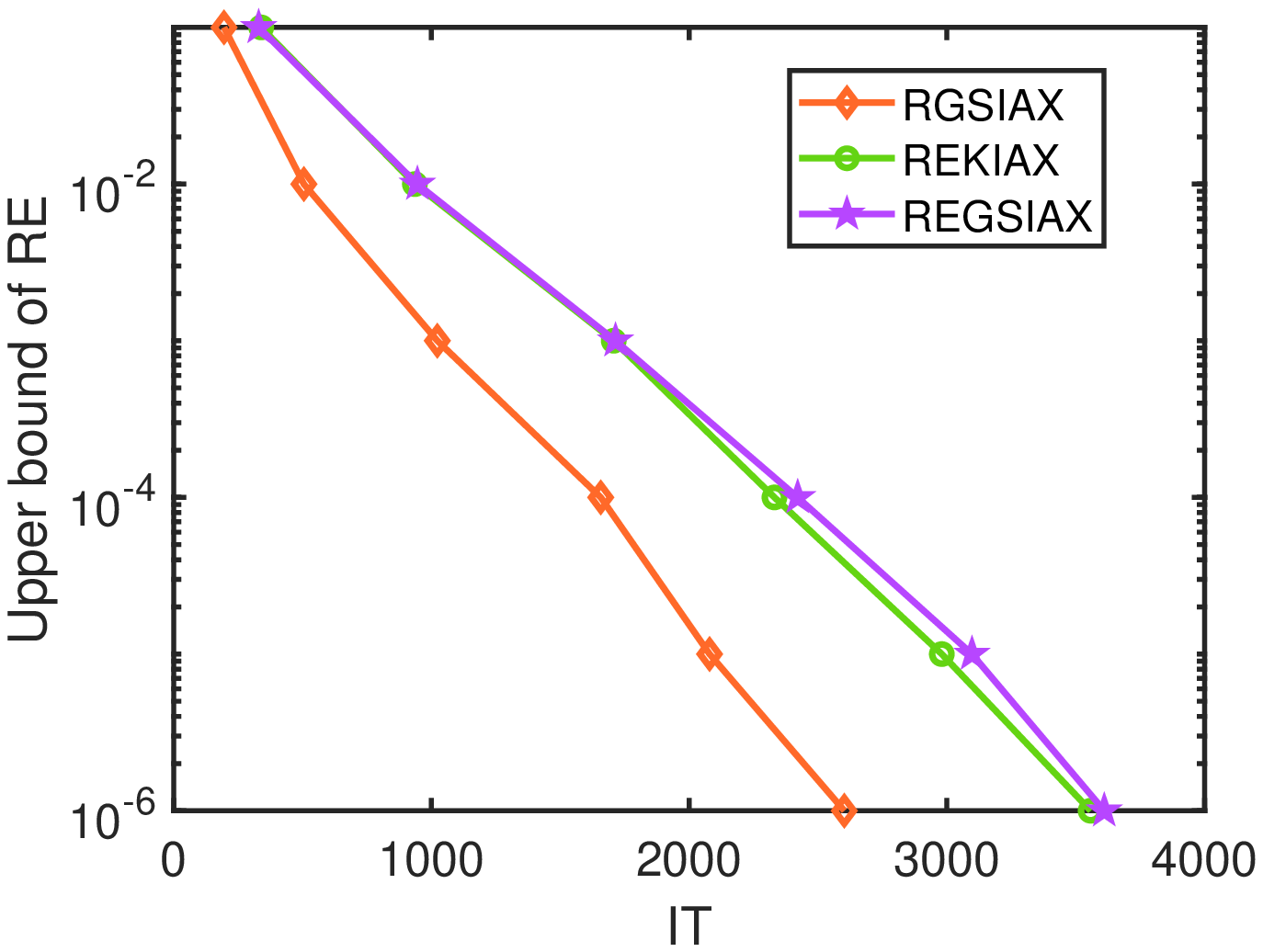} 		\includegraphics[width=4cm]{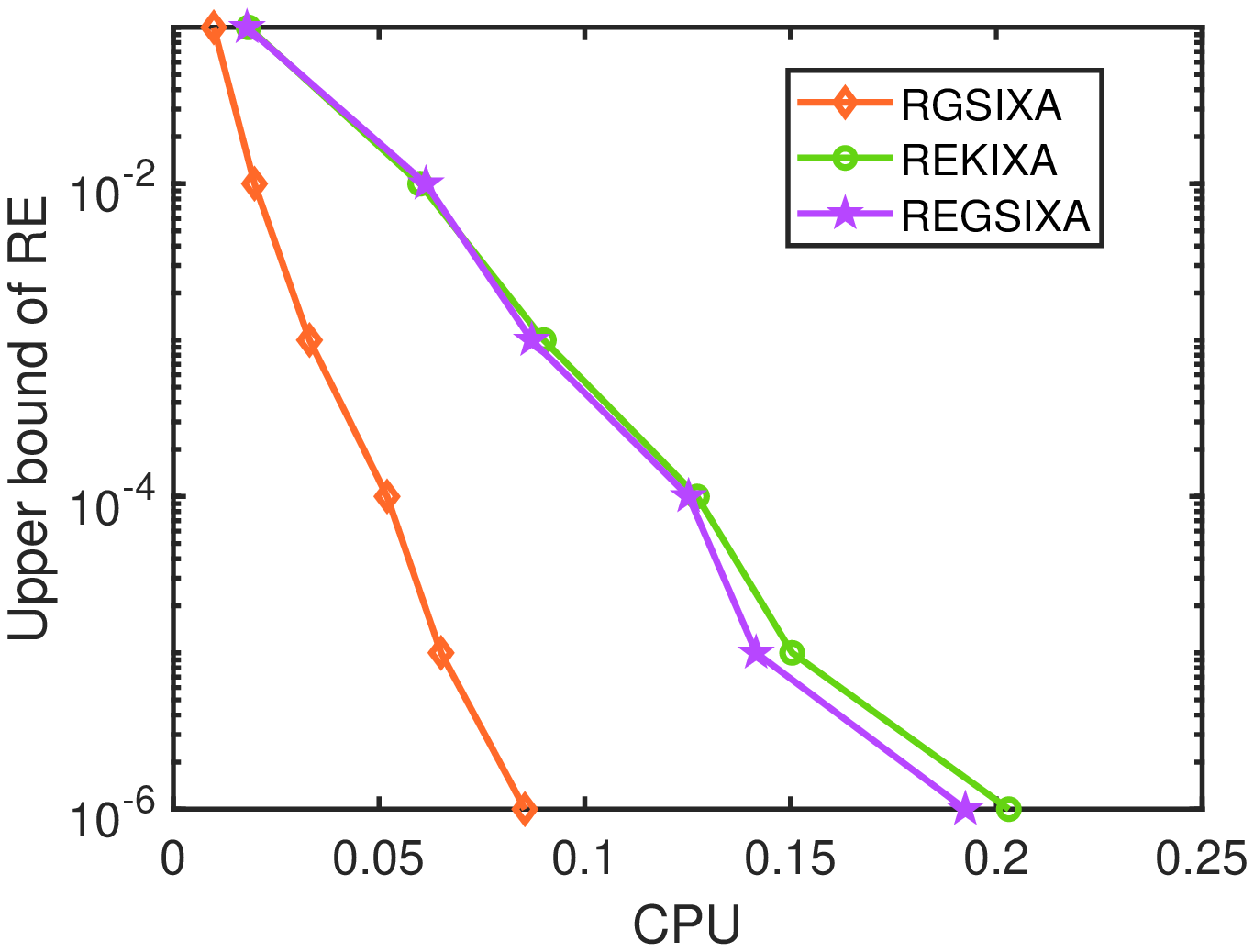}
	\includegraphics[width=4cm]{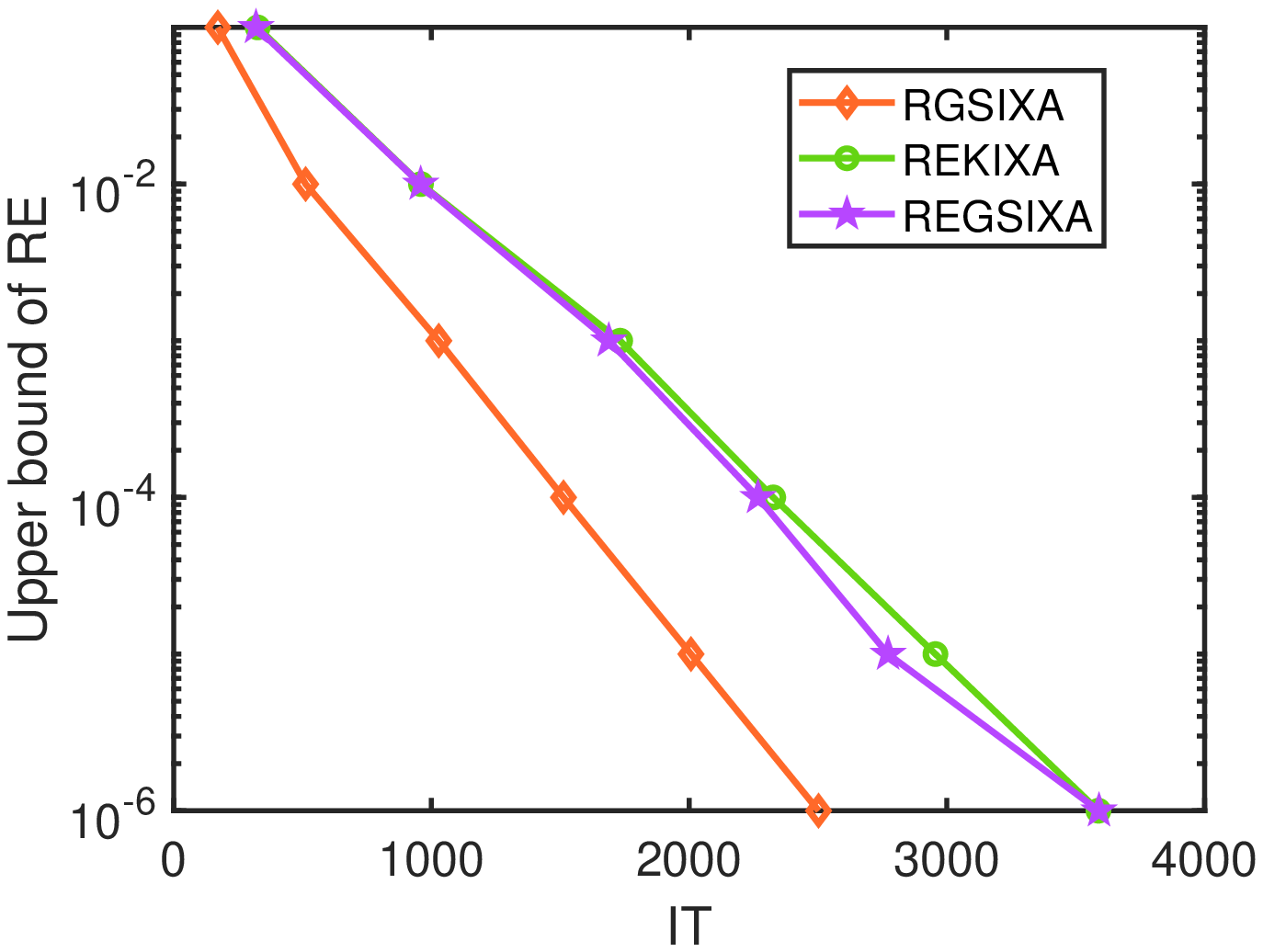}
	\caption{Relative errors of different methods for inconsistent matrix equations $AX=B$ and $XA=C$.}\label{figure:6.3.1}
\end{figure}

\begin{table}[htbp]
	\footnotesize
	\resizebox{\textwidth}{!}{
		\begin{threeparttable}
			\caption{The average CPU and IT of RKCAX, RKCXA, RGSIAX, RGSIXA, REKIAX, REKIXA, REGSIAX, and REGSIXA for solving consistent matrix equations for Type I.}
			\label{table:6.1.1}
			\begin{tabular}{ccccccccccccc}
				\toprule  
				m & n & p & rank(A) &     & RKCAX  & RKCXA   & RGSIAX   & RGSIXA & REKIAX   & REKIXA    & REGSIAX  & REGSIXA  \\
				\hline
				\multirow{2}*{50} & \multirow{2}*{30} & \multirow{2}*{30} & \multirow{2}*{30}
				& CPU  & 0.1236  & 0.1206   & 0.1329  & $-$   & 0.3499  & 0.3608  & 0.3416   & 0.3292     \\
				&  &  &    & IT   & 4163    & 4079     & 4439    & $-$   & 6085    & 6057    & 6157     & 6063    \\	
				\hline
				\multirow{2}*{30} & \multirow{2}*{50} & \multirow{2}*{30} & \multirow{2}*{30}
				& CPU  & 0.0511  & 0.0476   & $-$    & 0.0519   & 0.1301   & 0.1366   & 0.1324  & 0.1244     \\
				&  &  &    & IT   & 1915    & 1854     & $-$    & 1942     & 2686     & 2769     & 2725    & 2711   \\
				\hline
				\multirow{2}*{50} & \multirow{2}*{30} & \multirow{2}*{30} & \multirow{2}*{15}
				& CPU  & 0.0101   & 0.0097   & $-$     & $-$    & 0.0254  & 0.0259   & 0.0248  & 0.0232      \\
				&  &  &    & IT   & 357      & 347      & $-$     & $-$    & 510     & 501      & 503     & 474    \\
				\hline
				\multirow{2}*{100} & \multirow{2}*{60} & \multirow{2}*{60} & \multirow{2}*{60}
				& CPU  & 0.1902   & 0.2725   & 0.2118   & $-$     & 0.6184   & 0.6799   & 0.5406   & 0.5678    \\
				&  &  &    & IT   & 4717     & 4887     & 5111     & $-$     & 7092     & 7226     & 7134     & 7159   \\
				\hline
				\multirow{2}*{60} & \multirow{2}*{100} & \multirow{2}*{60} & \multirow{2}*{60}
				& CPU  & 0.1990   & 0.1944   & $-$     & 0.1913    & 0.5722    & 0.5558   & 0.5064    & 0.4534    \\
				&  &  &    & IT   & 4502     & 4382     & $-$     & 4690      & 6579      & 6560     & 6513      & 6499  \\
				\hline
				\multirow{2}*{100} & \multirow{2}*{60} & \multirow{2}*{60} & \multirow{2}*{30}
				& CPU  & 0.0237  & 0.0306    & $-$   & $-$      & 0.0719     & 0.0761    & 0.0618   & 0.0626    \\
				&  &  &    & IT   & 690     & 658       & $-$   & $-$      & 911        & 932       & 922      & 948    \\
				\hline
				\multirow{2}*{200} & \multirow{2}*{100} & \multirow{2}*{100} & \multirow{2}*{100}
				& CPU  & 0.4867  & 0.6250   & 0.5580   & $-$    & 1.2811   & 1.4384    & 1.3371   & 1.5524    \\
				&  &  &    & IT   & 4962    & 5044     & 5301     & $-$    & 7242     & 7262      & 7256     & 7357    \\
				\hline
				\multirow{2}*{100} & \multirow{2}*{200} & \multirow{2}*{100} & \multirow{2}*{100}
				& CPU  & 0.5993  & 0.4757   & $-$   & 0.5766    & 1.3872    & 1.2776    & 1.4842   & 1.3001    \\
				&  &  &    & IT   & 5100    & 5098     & $-$   & 5376      & 7384      & 7427      & 7397     & 7456    \\
				\hline
				\multirow{2}*{200} & \multirow{2}*{100} & \multirow{2}*{100} & \multirow{2}*{50}
				& CPU  & 0.1148  & 0.1446   & $-$    & $-$     & 0.3008    & 0.3148    & 0.2833   & 0.3134    \\
				&  &  &    & IT   & 1142    & 1101     & $-$    & $-$     & 1612      & 1565      & 1556     & 1558    \\
				\bottomrule 
			\end{tabular}
		\end{threeparttable}
	}
\end{table}

\begin{table}[htbp]
	\footnotesize
	\resizebox{\textwidth}{!}{
		\begin{threeparttable}
			\caption{The average CPU and IT of RKCAX, RKCXA, RGSIAX, RGSIXA, REKIAX, REKIXA, REGSIAX, and REGSIXA for solving consistent matrix equation for Type II.}
			\label{table:6.2.1}
			\begin{tabular}{ccccccccccccc}
				\toprule  
				A   &     & RKCAX  & RKCXA   & RGSIAX   & RGSIXA & REKIAX   & REKIXA    & REGSIAX  & REGSIXA  \\
				\hline
				\multirow{2}*{divorce}
				& CPU  & 0.0723  & 0.0754   & 0.0792  & $-$   & 0.1940  & 0.2006  & 0.2048   & 0.1973     \\
				& IT   & 2925    & 3059     & 3242    & $-$   & 4261    & 4251    & 4360     & 4261    \\	
				\hline
				\multirow{2}*{divorce$^{\top}$}
				& CPU  & 0.0950  & 0.0805   & $-$    & 0.0941   & 0.2300  & 0.2576   & 0.2473  & 0.2419     \\
				& IT   & 2921    & 2874     & $-$    & 3273     & 4298    & 4595     & 4117    & 4314    \\
				\hline
				\multirow{2}*{ash219}
				& CPU  & 0.0693   & 0.0898   & 0.0731   & $-$    & 0.1677  & 0.2183   & 0.1814  & 0.2509      \\
				& IT   & 1966     & 1790     & 2080     & $-$    & 2553    & 2562     & 2518    & 2838    \\
				\hline
				\multirow{2}*{ash219$^{\top}$}
				& CPU  & 0.0714   & 0.0613   & $-$    & 0.0761    & 0.1832   & 0.1792   & 0.1685   & 0.1840    \\
				& IT   & 2061     & 1883     & $-$    & 1908      & 2836     & 2753     & 2537     & 2550    \\
				\hline
				\multirow{2}*{Worldcities}
				& CPU  & 1.1540     & 2.2988    & 1.2106     & $-$    & 3.6571   & 6.1188    & 3.3258     & 4.9767    \\
				& IT   & 34738      & 38719     & 39211      & $-$    & 52273    & 56721     & 53722      & 56789   \\
				\hline
				\multirow{2}*{Worldcities$^{\top}$}
				& CPU  & 1.2753     & 1.2363    & $-$    & 1.9013     & 4.1061    & 4.3065    & 3.5750     & 4.7064   \\
				& IT   & 36365      & 38426     & $-$    & 40501      & 54252     & 56644     & 53114      & 57918   \\
				\hline
				\multirow{2}*{ash958}
				& CPU  & 0.3315     & 0.5544    & 0.3314     & $-$    & 0.9214    & 1.4136    & 0.9256     & 1.4477    \\
				& IT   & 6421       & 6002      & 5788       & $-$    & 8444      & 8402      & 8498       & 8010   \\
				\hline
				\multirow{2}*{ash958$^{\top}$}
				& CPU  & 0.9026     & 0.8464    & $-$     & 0.8674     & 1.9521    & 2.4070    & 2.1676     & 2.0962    \\
				& IT   & 6283       & 6133      & $-$     & 5785       & 8758      & 8348      & 8007       & 8216   \\
				\bottomrule 
			\end{tabular}
		\end{threeparttable}
	}
\end{table}

\begin{table}[htbp]
	\footnotesize
	\resizebox{\textwidth}{!}{
		\begin{threeparttable}
			\caption{The average CPU and IT of RGSIAX, RGSIXA, REKIAX, REKIXA, REGSIAX, and REGSIXA for solving inconsistent matrix for Type I.}
			\label{table:6.3.1}
			\begin{tabular}{ccccccccccccc}
				\toprule  
				m & n & p & rank(A) &               & RGSIAX   & RGSIXA & REKIAX   & REKIXA     & REGSIAX  & REGSIXA  \\
				\hline
				\multirow{2}*{50} & \multirow{2}*{30} & \multirow{2}*{30} & \multirow{2}*{30}
				& CPU     & 0.0742   & $-$    & 0.1833  & 0.1892   & 0.1977  & 0.1959      \\
				&  &  &    & IT      & 2602     & $-$    & 3558    & 3511     & 3611    & 3585    \\
				\hline
				\multirow{2}*{30} & \multirow{2}*{50} & \multirow{2}*{30} & \multirow{2}*{30}
				& CPU     & $-$    & 0.0854    & 0.1900   & 0.2030   & 0.2072   & 0.1925    \\
				&  &  &    & IT      & $-$    & 2502      & 3402     & 3588     & 3489     & 3590    \\
				\hline
				\multirow{2}*{50} & \multirow{2}*{30} & \multirow{2}*{30} & \multirow{2}*{15}
				& CPU     & $-$     & $-$   & 0.0291  & 0.0350  & 0.0365   & 0.0333     \\
				&  &  &    & IT      & $-$     & $-$   & 575     & 608     & 616      & 604    \\	
				\hline
				\multirow{2}*{100} & \multirow{2}*{60} & \multirow{2}*{60} & \multirow{2}*{60}
				& CPU     & 0.1971   & $-$    & 0.5141   & 0.6265    & 0.5439   & 0.5899   \\
				&  &  &    & IT      & 4952     & $-$    & 6888     & 6993      & 7013     & 6918    \\
				\hline
				\multirow{2}*{60} & \multirow{2}*{100} & \multirow{2}*{60} & \multirow{2}*{60}
				& CPU     & $-$   & 0.4803    & 0.9187    & 0.9225    & 0.8883   & 0.7187    \\
				&  &  &    & IT      & $-$   & 5628      & 9144      & 9093      & 9221     & 9146    \\
				\hline
				\multirow{2}*{100} & \multirow{2}*{60} & \multirow{2}*{60} & \multirow{2}*{30}
				& CPU     & $-$      & $-$    & 0.0987   & 0.0983    & 0.0713   & 0.0871    \\
				&  &  &    & IT      & $-$      & $-$    & 923      & 946       & 907      & 900    \\
				\hline
				\multirow{2}*{200} & \multirow{2}*{100} & \multirow{2}*{100} & \multirow{2}*{100}
				& CPU     & 0.5820   & $-$     & 1.3477    & 1.5606    & 1.4141   & 1.6983    \\
				&  &  &    & IT      & 5239     & $-$     & 7372      & 7356      & 7403     & 7343    \\
				\hline
				\multirow{2}*{100} & \multirow{2}*{200} & \multirow{2}*{100} & \multirow{2}*{100}
				& CPU     & $-$   & 1.1977    & 1.4616     & 1.3366   & 1.6726   & 1.6002    \\
				&  &  &    & IT      & $-$   & 5450      & 7624       & 7500     & 7618     & 7550    \\
				\hline
				\multirow{2}*{200} & \multirow{2}*{100} & \multirow{2}*{100} & \multirow{2}*{50}
				& CPU     & $-$    & $-$    & 0.2044   & 0.2402    & 0.2080  & 0.2486   \\
				&  &  &    & IT      & $-$    & $-$    & 1390     & 1415      & 1370    & 1423    \\
				\bottomrule 
			\end{tabular}
		\end{threeparttable}
	}
\end{table}

\begin{table}[htbp]
	\footnotesize
	\resizebox{\textwidth}{!}{
		\begin{threeparttable}
			\caption{The average CPU and IT of RGSIAX, RGSIXA, REKIAX, REKIXA, REGSIAX, and REGSIXA for solving inconsistent matrix equation for Type II.}
			\label{table:6.4.1}
			\begin{tabular}{ccccccccccccc}
				\toprule  
				A   &       & RGSIAX   & RGSIXA & REKIAX   & REKIXA    & REGSIAX  & REGSIXA  \\
				\hline
				\multirow{2}*{divorce}
				& CPU     & 0.0965  & $-$   & 0.2182  & 0.2145  & 0.2186   & 0.2124     \\
				& IT      & 3128    & $-$   & 4474    & 4316    & 4406     & 4228    \\	
				\hline
				\multirow{2}*{divorce$^{\top}$}
				& CPU    & $-$    & 0.0835   & 0.2218  & 0.2192   & 0.2424  & 0.2086     \\
				& IT     & $-$    & 3111     & 4507    & 4319     & 4531    & 4329    \\
				\hline
				\multirow{2}*{ash219}
				& CPU    & 0.0608   & $-$    & 0.1620  & 0.2267   & 0.1843  & 0.2111      \\
				& IT     & 1974     & $-$    & 2698    & 2818     & 2868    & 2616    \\
				\hline
				\multirow{2}*{ash219$^{\top}$}
				& CPU    & $-$    & 0.1069    & 0.1794   & 0.1544   & 0.1739   & 0.1883    \\
				& IT     & $-$    & 1893      & 2882     & 2527     & 2728     & 2508    \\
				\hline
				\multirow{2}*{Worldcities}
				& CPU   & 1.3091    & $-$    & 3.6970    & 5.9668    & 3.8696     & 5.2018    \\
				& IT    & 39696     & $-$    & 56191     & 54984     & 57586      & 55683    \\
				\hline
				\multirow{2}*{Worldcities$^{\top}$}
				& CPU   & $-$    & 3.5204    & 4.8339    & 4.3504    & 4.2900     & 4.9784    \\
				& IT    & $-$    & 3.9962    & 57669     & 55742     & 57501      & 55740    \\
				\hline
				\multirow{2}*{ash958}
				& CPU   & 0.3721    & $-$    & 0.9404    & 1.5042    & 0.9446     & 1.6585    \\
				& IT    & 5948      & $-$    & 8126      & 8339      & 8045       & 8336    \\
				\hline
				\multirow{2}*{ash958$^{\top}$}
				& CPU   & $-$    & 3.4370    & 1.3624    & 1.3549    & 1.6021     & 1.4762    \\
				& IT    & $-$    & 6112      & 8297      & 8424      & 8831       & 8381    \\
				\bottomrule 
			\end{tabular}
		\end{threeparttable}
	}
\end{table}	
\section*{Disclosure statement}

No potential conflict of interest was reported by the authors.

\bibliographystyle{tfs}
\bibliography{interacttfssample}

\end{document}